\documentclass[12pt]{amsart}
\usepackage{amssymb}
\usepackage{amsthm}
\usepackage{cite}
\usepackage{hyperref}
\usepackage{amsmath,mathtools}
\usepackage{amscd,enumitem}
\usepackage{caption}
\usepackage[justification=centering]{caption}
\usepackage{verbatim}
\usepackage{color}
\usepackage{geometry}
\usepackage{tikz}
\usetikzlibrary{decorations.markings}

\newtheorem*{conj*}{Conjecture}
\newtheorem{theorem}{Theorem}[section]

\theoremstyle{definition}
\newtheorem*{remark}{Remark}

\theoremstyle{plain}
\newtheorem{cor}[theorem]{Corollary}
\newtheorem{lemma}[theorem]{Lemma}

\newtheorem{prop}[theorem]{Proposition}

\newtheorem{definition}[theorem]{Definition}

\newcommand{\Z}{\mathbb{Z}}

\newcommand{\Q}{\mathbb{Q}}
\newcommand{\R}{\mathbb{R}}

\renewcommand{\pmod}[1]{\,\,({\rm mod}\,\,{#1})}
\numberwithin{equation}{section}

\usepackage{setspace}

\begin{document}

\onehalfspacing

\title{Quantum Modular Forms and Resurgence}
\author{Eleanor McSpirit}
\address{Department of Mathematics, Vanderbilt University, Nashville, TN 37240}
\email{eleanor.mcspirit@vanderbilt.edu}
\author{Larry Rolen}
\email{larry.rolen@vanderbilt.edu}

\subjclass[2020]{}

\begin{abstract} 

In 2010, Zagier described a new phenomenon which he called quantum modularity. This connected various examples coming from disparate fields which exhibit near-modular behavior.
In the fifteen years since, Zagier's philosophy has informed new developments in areas such as knot theory, 3-dimensional topology, combinatorics, and physics. More recently, the concept of holomorphic quantum modularity has emerged, pointing to a clearer structure for Zagier's original examples. These new developments suggest connections to perturbative quantum field theory, like the theory of resurgence. In 2024, Fantini and Rella proposed a means of 
codifying some of these connections under their program of ``modular resurgence." 

Inspired by their work, we unify all of the  examples of quantum modular forms in Zagier's original paper under the umbrella of resurgence. In doing so, we strengthen known quantum modularity results for holomorphic Eichler integrals of half-integer weight modular forms. Our main addition to the literature is a collection of median resummation type results which show that the examples of holomorphic quantum modular forms we consider can be recovered from their asymptotics, or the asymptotics of closely related functions. 
\end{abstract}

\

\maketitle

\section{Introduction and statement of results}

In a seminal 2010 paper \cite{zagier-quantum}, Zagier defined a new type of modular object which he called a quantum modular form. This new notion was meant to capture the behavior of a collection of examples from the literature. As Zagier remarked, these objects are reminiscent of calculations in perturbative quantum field theory and quantum topology. In the past few years, quantum modular forms have
appeared prominently in connection to exciting new invariants in low-dimensional topology and knot theory (e.g. \cite{gppv,g-z-24,ajk}).
In this paper, we utilize tools popularized by quantum field theory in order to unify the examples from Zagier's original paper and subsequent works under a more robust framework. 

Roughly speaking, a quantum modular form is a function which is not modular, but has an obstruction to modularity which is ``nice." In Zagier's original definition, these are not functions on the upper half-plane $\mathbb{H}$, but are rather defined on $\mathbb{P}^1(\mathbb{Q}) = \mathbb{Q} \cup \{i\infty\}$. In a typical example, the original function has no continuous extension to $\mathbb{P}^1(\mathbb{R})$, but the error to modularity extends to, for example, a $C^j$, $C^\infty$, or real-analytic function on $\mathbb{P}^1(\mathbb{R})$. However, the most mysterious example in Zagier's original paper, called $\mathbf{J}(x)$ and defined using the colored Jones polynomials of the figure-eight knot, had an obstruction to modularity which did not even extend to a continuous function. However, he included intriguing graphs of this obstruction that were suggestive of some underlying quantum modular structure.  

Over the last few years, the notion of holomorphic quantum modularity introduced by Garoufalidis and Zagier \cite{g-z-23, g-z-24} has given a new perspective on $\mathbf{J}(x)$. It turns out that Zagier's original function naturally embeds as one component of a matrix-valued function. When considered as a matrix, the corresponding error to modularity is $C^\infty$, clarifying the behavior of the original function's obstruction. 
A second development in their work was the full consideration of quantum modular forms as functions on $\mathbb{C}\backslash \mathbb{R}$ instead of $\mathbb{Q}$. In Zagier's original paper, he introduced ``strong" quantum modular forms, which to each rational number assigned a full formal power series instead of just one complex number, and noted that in each case these functions also admitted a full extension to $(\mathbb{C}\backslash \mathbb{R}) \cup \mathbb{Q}$.  This extension of the original definition goes hand in hand with the connection of resurgence to their work. 
 
Resurgence is a theory built on Borel resummation, a procedure for summing divergent series which has found a prominent foothold in quantum field theory. This was naturally connected to $\mathbf{J}(x)$ due to the deep relationship between Jones polynomials and Chern--Simons theory, first realized in \cite{witten}. This connection has only deepened as resurgence and modularity have informed the study of quantum invariants of knots and 3-manifolds, see e.g. \cite{garoufalidis, 3d-modularity, gukov-marino-putrov, andersen-mistegard}.

This connection to resurgence naturally led to the consideration of functions defined beyond $\mathbb{Q}$ which recovered the original quantum modular form through asymptotics. This new breakthrough and similar examples inspired the introduction of matrix-valued holomorphic quantum modular forms. In keeping with modern definitions such as in \cite{zagiertalk}, here we offer a definition in the scalar-valued case which will suit our present purposes. We define the slash operator for the weights considered in this paper in Equation \ref{eqn: slash}.
\begin{definition}
	We say that a holomorphic function $g \colon \mathbb{H} \to \mathbb{C}$ is a holomorphic quantum modular form of weight $\kappa \in \mathbb{C}$ for $\Gamma_0(N)$ if for all $\gamma =\begin{psmallmatrix} a & b \\ c & d \end{psmallmatrix} \in \Gamma_0(N)$, we have that 
	\[h_\gamma(z):= g(z) - g|_{\kappa }\gamma(z) 
	\]
	extends to an analytic function on the cut plane $\displaystyle{\mathbb{C}_\gamma := \{ z \in \mathbb{C} : cz+d \in \mathbb{C}\backslash\mathbb{R}_{\leq 0}\}.}$
\end{definition}

While these developments gave a richer understanding of the quantum modularity of $\mathbf{J}(x)$, they did not make clear how the rest of the examples in Zagier's original paper might relate to both holomorphic quantum modularity and resurgence.

Recently, Fantini and Rella made progress on this front. In \cite{fantini-rella}, the authors offered a number of conjectures supported by examples about how resurgence can generate holomorphic quantum modular forms given the extra data of an $L$-function with a functional equation. To describe this phenomenon, they coined the term ``modular resurgence."

We see their work as the start of a more general program to make the connection between holomorphic quantum modularity and resurgence explicit. 
In this paper, we continue to investigate this connection, unifying the examples in Zagier's original paper under one umbrella. In particular, we show that well-known families of quantum modular forms can be recovered by resurgence, and that resurgence naturally gives rise to quantum modularity. 

\subsection{Statement of Main Results} Before we give the general statement of Theorem \ref{thm: Main}, we highlight corollaries in some key special cases which correspond to families of quantum modular forms generalizing Zagier's original examples. 
We note for the reader that the definitions of various terms from the theory of resurgence are deferred until Section 2.1 for ease of exposition.

The first example we consider originally appeared in work of Lawrence and Zagier \cite{lawrence-zagier} and was subsequently named Example 4 in \cite{zagier-quantum}. This example was born out of consideration of the Witten--Reshetikhin--Turaev (WRT) invariants of 3-manifolds. Once a manifold has been fixed, this collection of invariants define a function from roots of unity to $\mathbb{C}$.  In the case where the 3-manifold is the Poincar\'e homology sphere, Lawrence and Zagier
showed that this function, after suitable rescaling, is a quantum modular form. 

Moreover, they proved that there exists a $q$-series converging for $|q|<1$ which recovers (up to the same rescaling) the associated WRT invariants by taking the limiting values as $q$ tend to roots of unity radially. This function $\widetilde{\Theta}_+(q)$ has coefficients given by
\[\chi_+(n) := \begin{cases}
 	(-1)^{\lfloor n/30\rfloor} & \text{if } (n, 6) = 1 \text{ and } n \equiv \pm 1 \pmod{5},\\
 	0 & \text{otherwise,}
 \end{cases}
\]
so that
\[\widetilde{\Theta}_+(q) = \sum_{n=1}^\infty \chi_+(n)q^{n^2/120}.
\]
This $q$-series bears a strong resemblance to the weight $3/2$ modular form
\[\Theta_+(z) := \sum_{n=1}^\infty n\chi_+(n)q^{n^2/120}, \quad \quad (q = e^{2\pi i z}, z \in \mathbb{H}).
\]
One can view $\Theta_+(z)$ as essentially a ``half $q$-derivative" of $\widetilde{\Theta}_+$. 
This relationship was made precise by extending the notion of an Eichler integral (see Section \ref{subsec: modular} or\cite{eichler, shimura, manin} for more details) of a cusp form to the half-integer weight case.

In \cite{zagier-quantum}, other Eichler integrals of cusp forms were given as Examples 2 and 3 of quantum modular forms on $\mathbb{Q}$, illustrating a more general phenomenon.  
In \cite{bringmann-rolen}, Bringmann and the second author proved that this procedure produces quantum modular forms when one starts with a half-integral weight cusp form. 

Extending work of Bringmann and the second author which established the quantum modularity of Eichler integrals, we use techniques from resurgence to prove that these Eichler integrals are in fact holomorphic quantum modular forms. The definition of an Eichler integral is given in Section \ref{subsec: modular}, and here $\chi_{-4}(n) := \left( \frac{-4}{n}\right)$ denotes the Kronecker symbol. 

\begin{cor}\label{cor: eichler}
Let $\tilde{f}(z)$ be the Eichler integral of a cusp form of weight $k \in \frac{1}{2} + \mathbb{N}$ for $\Gamma_0(N)$, where $4 \mid N$. For each $\gamma =\begin{psmallmatrix} a & b \\ c & d \end{psmallmatrix} \in \Gamma_0(N)$, define
\begin{align*}
	\tilde{h}_\gamma(z):= \tilde{f}(z)  -  \chi_{-4}(d)\tilde{f}|_{2-k}\gamma(z). 
\end{align*}
	Then the following are true.
\begin{enumerate}
\item The function $\tilde{h}_\gamma(z)$ has an analytic continuation to $\mathbb{C}_\gamma$. 
\item We have the integral representation \[\displaystyle{\tilde{h}_\gamma(z) = \frac{(-2\pi i)^{k-1}}{\Gamma(k-1)} \int_{-d/c}^{i\infty} f(\tau)\left(\tau-z \right)^{k-2}\;d\tau},\] where both sides are defined, and the path of integration avoids the branch cut. 
\end{enumerate}
\end{cor}

We note that this corollary ties the half-integral weight case more directly to the integral weight case; for an Eichler integral of an integral weight cusp form, the expression in \emph{(2)} is the corresponding period polynomial. We show this relationship explicitly in Section \ref{sec: corollaries}. 

We also remark that related results were previously attained for specific families of Eichler integrals of weight $1/2$ in \cite{gukov-marino-putrov, andersen-mistegard, 3d-modularity, hanetal, bringmann-nazaroglu}.

Looking at another family of quantum modular forms, we recall Example 1 from \cite{zagier-quantum}. The starting point is the function $\sigma$ from Ramanujan's Lost Notebook:
\[\sigma(q) := \sum_{n=0}^\infty \frac{q^{n(n+1)/2}}{(1+q)(1+q^2)\cdots(1+q^n)}, \quad \quad (|q|<1). 
\]
This function has a natural ``companion" $\sigma^*$ defined instead for $|q|>1$. 

The function $\sigma(q)$ garnered considerable interest following several conjectures of Andrews \cite{andrews}. Notably, work of Andrews, Dyson, and Hickerson \cite{andrews-dyson-hickerson} expressed $\sigma$ and $\sigma^*$ as indefinite theta series, and work of Cohen \cite{cohen} realized coefficients of $\sigma$ and $\sigma^*$ as the Dirichlet coefficients of an Artin $L$-function. 
Moreover, authors of both papers established that $\sigma$ has a $q$-series expansion that converges not only for $|q|<1$, but is finite whenever $q$ is a root of unity, leading to a natural definition of $\sigma(e^{2\pi i x})$ for $x \in \mathbb{Q}$. With the addition of a phase factor, we define 
 \[f(x) := q^{1/24} \sigma(q), \quad \quad(q = e^{2\pi i x}, x\in \mathbb{Q}).\]
It then turns out that $f$ is a quantum modular form on $\mathbb{Q}$ for $\Gamma_0(2)$. 

Work of Lewis and Zagier offered a framework which generalized many of the interesting properties of $\sigma$. In \cite{lewis-zagier}, they showed a correspondence between Maass waveforms $u$, pairs of $L$-functions $(L_0,L_1)$, holomorphic periodic functions $f$, and period functions $\psi$. For the precise definitions and basic results on Maass waveforms are given in Section \ref{subsec: maass}. This period function appears as the obstruction to modularity of $f$ under $z \mapsto -1/z$ and extends to the cut plane $\mathbb{C}':=\mathbb{C}\backslash (-\infty,0]$, upgrading the first observation of quantum modularity of $\sigma$ to a full holomorphic quantum modularity result. We use resurgence to give a new proof of this.

\begin{cor}
Let $f(z)$ be the periodic function of a Maass waveform $u(z)$ with spectral parameter $s= \frac{1}{2}$ for $\mathrm{SL}_2(\mathbb{Z})$. Define
\begin{align*}
	\psi(z):= f(z)  -  z^{-1}f(-1/z). 
\end{align*}
	Then the function $\psi$ has an analytic continuation to $\mathbb{C}'$. 
\end{cor}

We note that \cite{fantini-rella} also offered a proof of this result using resurgence. However, the main portion of our proof, a median resummation result using Gaussian integration, is distinct. In particular, we show that one does not need to assume any information about the period function proven in \cite{lewis-zagier}. 

We mention one last type of quantum modular form, which Zagier labelled as Example 0 in \cite{zagier-quantum}. To describe this, for real numbers $x$ consider the sawtooth function $((x))$, which takes the value $0$ for integers $x$ and is defined by $((x)):=x-\lfloor x\rfloor-1/2$ for $x\in\R\setminus\Z$. Then the classical Dedekind sum for coprime integers $c,d$ with $c>0$ is defined by 
\[
s(d,c):=\sum_{0<k<c}\left(\left(\frac{k}c\right)\right)\cdot\left(\left(\frac{kd}c\right)\right).
\]
These Dedekind sums satisfy well-known relations, in particular, reciprocity. Zagier interpreted these by defining $S\colon\Q\rightarrow\Q$ by $S(d/c):=12s(d,c)$, and then noted that they imply the quantum modular-like transformations
\[
S(x)-S(x+1)=0,\quad\quad S(x)-S(-1/x)=x+\frac1x\pm 3+\frac{1}{\mathrm{Num}(x)\mathrm{Den}(x)}.
\]
Here, the sign above is $+$ if $x<0$, and is $-$ if $x>0$. As these obstructions to modularity depend on numerators and denominators, they are not continuous functions. However, such examples found key applications in work of Bettin and Conrey \cite{bettin-conrey}. Later, in \cite{folsom-cotangent-sums} Folsom embedded these examples into an infinite family of twisted versions of Maass--Eisenstein series. Generically, she proved that they are quantum modular forms with smooth errors to modularity. However, when one specializes to the examples of Zagier, or of Bettin and Conrey, these smooth functions ``degenerate'' to discontinuous functions. Thus, even this example from Zagier's original paper fits into a rich theory of period functions of Maass--Eisenstein series. 

With some mild extra considerations specific to the non-cuspidal case, we demonstrate how resurgence connects to Maass--Eisenstein series with the following corollary. 

\begin{cor}
Let $f(z)$ be the periodic function of the real-analytic Eisenstein series $E_{1/2}(z)$, defined in Equation $\ref{eqn: maass eisenstein}$. Define
\begin{align*}
	\psi(z):= f(z)  -  z^{-1}f(-1/z). 
\end{align*}
	Then the function $\psi$ has an analytic continuation to $\mathbb{C}'$. 
	\end{cor}

We remark that our methods are likely generalizable to more general spectral parameter and level for both Maass cases. 

As an additional result showing the flexibility of these methods, we reprove the quantum modularity of cusp forms and of Eichler integrals of integer weight cusp forms using a variant of the same framework. We remark that in these cases, resurgence analysis on its face will not apply; see Remark 3.2 of \cite{fantini-rella}. 
Our main result is the following:

\begin{theorem} \label{thm: Main} 
Suppose $g$ is one of the following:
\begin{enumerate}
\item a cusp form of weight $\kappa \in \mathbb{N}$ for $\Gamma =\Gamma_0(N)$;
\item a cusp form of weight $\kappa \in \frac{1}{2} + \mathbb{N}_0$ for $\Gamma =\Gamma_0(N)$, $4\mid N$;
\item the Eichler integral of a cusp form of weight $2-\kappa \in \mathbb{N}$ for $\Gamma =\Gamma_0(N)$;
\item the Eichler integral of a cusp form of weight $2-\kappa \in \frac{1}{2} + \mathbb{N}_0$ for $\Gamma =\Gamma_0(N)$, $4\mid N$;
\item the periodic function of a Maass waveform with $s=\frac{1}{2} = \kappa/2$ for $\Gamma =\text{SL}_2(\mathbb{Z})$;
\item the periodic function of $E_{\kappa/2}(z)$ with $s=\frac{1}{2}=\kappa/2$ for $\Gamma =\text{SL}_2(\mathbb{Z})$.
\end{enumerate}
Then $g$ is a holomorphic quantum modular form of weight $\kappa$ for $\Gamma$. Moreover, for each $\gamma \in \Gamma$, the modular obstruction $h_\gamma$ equals a Borel--Laplace sum of the asymptotic series of $g$ as $z$ approaches $\gamma^{-1}(i\infty)$ along a geodesic, and analytically continues to $\mathbb{C}_\gamma$.
\end{theorem}

For the remainder of the paper, we will refer to these cases as \emph{Type 1} through \emph{Type 6}. We note that the results above for \emph{Types 1} and \emph{2} are immediate from their definitions. However, the methods developed for the remaining types in this paper can be applied to functions related to those of \emph{Types 1} and \emph{2} to give an interesting resurgence perspective on these functions which themselves have no resurgent structure.

\begin{remark} 
Many of the above functions have a natural ``companion" in the lower half-plane like the example of $\sigma$ and $\sigma^*$. Our results can equivalently be proven for the corresponding functions on $\mathbb{C}\backslash\mathbb{R}$. 
\end{remark}

We note that the quantum modularity of these functions is not new, but the quantum modularity of half-integral weight Eichler integrals has been strengthened to holomorphic quantum modularity, and the relationship of Borel resummation to their quantum modularity across all cases was previously conjectured but unproven.

Theorem~\ref{thm: Main} and its proof additionally give evidence towards some conjectures of Fantini and Rella \cite{fantini-rella}:

\begin{cor} \label{cor: conj}
The asymptotic series of Eichler integrals of half-integral weight cusp forms, periodic functions of Maass waveforms with $s=1/2$, and the periodic function of the real-analytic Eisenstein series with $s=1/2$ are modular resurgent series. Moreover, the median resummation of each of these series recovers the original function, which is a holomorphic quantum modular form. 
\end{cor}


The paper is organized as follows. In Section \ref{sec: prelims}, we develop the necessary background on resurgence, as well as holomorphic modular and Maass forms. In Section \ref{sec: intermediate}, we collect a variety of intermediate results in service of our main theorem. 
This will be broken up into three main subsections. 
Subsection 3.1 will compute the asymptotic expansions of functions of \emph{Types 1} through \emph{6} at the relevant rational points. Subsection 3.2 will use a Gaussian integration identity to relate each function to the median resummation of its asymptotic series. Subsection 3.3 will locate the image of each function under the relevant slash operator in the decomposition coming from Subsection 3.2.
In Section \ref{sec: main theorem}, we prove Theorem \ref{thm: Main} and in Section \ref{sec: corollaries}, we prove that the obstructions to modularity have the desired properties for the statements of our corollaries. In Section \ref{sec: questions}, we offer some future directions and open questions that may be of interest to the reader. 

\section*{Acknowledgements}
EM acknowledges the support of an AMS-Simons Travel Grant.  LR was supported by a grant from the Simons Foundation (853830, LR).

\section{Preliminaries}\label{sec: prelims}

\subsection{Resurgence} \label{subsec: resurgence}

Resurgence is a theory of divergent series developed by \'Ecalle in the 1980s, building on work dating back to Borel. It begins with classical Borel resummation, but leverages how exponentially small corrections to Borel-Laplace sums ``resurge" by expanding around different singularities of the Borel transform. For a more thorough recollection of Borel resummation and resurgence, see \cite{Balser, Costin, Sauzin}.

The initial objects of interest in the theory of resurgence are formal power series
\[
\varphi(z) = \sum_{n =0}^\infty a_n z^{n} \in \mathbb{C}[[z]].
\]
A rich source of such series are asymptotic series; 
we write
$\displaystyle{f(z) \sim \varphi(z)}$
if 
\[
f(z) - \sum_{n=0}^{N-1}a_nz^n = O(z^N)
\]
as $z \to 0$ for all $N \geq 1$. If this is the case, we say that $\varphi(z)$ is an asymptotic expansion of $f(z)$ at $z=0$. For our purposes, we will consider asymptotic expansions where $z = \frac{it}{2\pi}$ and $t \to 0^+$. 
We will be interested in factorially divergent series, as characterized by the following definition.

\begin{definition}
	We say that a formal power series $\varphi(z)$ is \textit{Gevrey-1} if there exist $A, B >0$ such that
	\[
	|a_n| \leq AB^n\Gamma(n+1), \quad  \quad (n \geq 0),
	\]
	where $\Gamma$ is the gamma function.
\end{definition}

Here, we note that $\Gamma(n+1)$ equals $n!$ since $n \in \mathbb{N}$. We choose to use the gamma function throughout the paper, even when the factorial suffices, since it will allow for a uniform treatment of all of our cases.

In order to build the theory of resurgence, we will first need to define the Borel and Laplace transforms. Here, we introduce shifted versions of these transforms which will more readily suit our needs. 

\begin{definition} The (formal) Borel transform with shift $\kappa  \in \mathbb{C}$ is defined on formal power series by  
\[\mathcal{B}_{\kappa}\left(\sum_{n=0}^\infty a_n z^n\right) := \sum_{n=0}^\infty \frac{a_n}{\Gamma(n+\kappa)}u^n.
\]
\end{definition}
When $\kappa  = 1$, we simply call this the Borel transform, dropping the subscript. 
The subsequent lemma can be readily proven from these definitions. 

\begin{lemma}
If $\varphi(z)$ is a Gevrey-1 formal power series, its Borel transform is analytic in a neighborhood of $u=0$.	
\end{lemma}

We also require the Laplace transform. We will consider integral transforms of functions $\psi(u)$ that are analytic on a half-strip $\{u \in \mathbb{C} \colon \text{dist}(u, e^{i\theta}\mathbb{R}_+) < \varepsilon\}$ with exponential type less than some $\alpha \in \mathbb{R}$, i.e. there exists an $M$ such that $|\psi(re^{i\theta})| \leq Me^{\alpha r}$ as $r$ tends to infinity.

\begin{definition} Suppose $\psi(u)$ is analytic on a half-strip $\{u \in \mathbb{C} \colon \emph{dist}(u, e^{i\theta}\mathbb{R}_+) < \varepsilon\}$ with exponential type less than some $\alpha \in \mathbb{R}$. Then the Laplace transform with shift $\kappa$ in direction $\theta$, when it exists, is defined by 
\[\mathcal{L}_{\kappa}^\theta (\psi)(z) := z^{-\kappa} \int_0^{e^{i\theta}\infty}e^{-u/z} \psi(u) u^{\kappa-1}\;du.
\]
\end{definition}
With the given conditions on $\psi$, the function $\mathcal{L}_{1}^\theta(\psi)(z)$ is defined and analytic in the right half-plane $\{\Re(ze^{-i\theta})> \max(\alpha, 0)\}$, and is locally analytic as one varies $\theta$, as long as the directions remain in the sector where $\psi$ is analytic. As with the Borel transform, we drop the subscript when $\kappa=1$.  

We note that instead of defining the $\kappa$-shifted Borel and Laplace transforms, which we do to simplify exposition, one can multiply the underlying asymptotic series by $z^{\kappa}$ and proceed accordingly.

\begin{definition}
	The composition $\mathcal{L}_{\kappa}^\theta \mathcal{B}_{\kappa}$ is called the \textit{Borel-Laplace sum} with shift $\kappa$ in direction $\theta$, which we will denote by $\mathcal{S}_\kappa^\theta$.  If the Borel-Laplace sum converges for a formal power series $\varphi$, 
we say that $\varphi$ is 1-summable in direction $\theta$.
\end{definition}
The Borel and Laplace transforms are formally inverses on monomials $z^n$ with $\Re(n+\kappa) > 0$, and if $\varphi$ is a series which converges in a neighborhood of 0, then $\varphi = \mathcal{L}^{\theta}_{\kappa}\mathcal{B}_{\kappa}(\varphi)$ for all $\theta$ and for all $\kappa$ such that $\Re(\kappa) > 0$.
Despite the fact that in general $\varphi \neq \mathcal{L}^{\theta}_{\kappa}\mathcal{B}_{\kappa}(\varphi)$, the Borel--Laplace sum of a Gevrey-1 divergent series is of interest to us since $\varphi$ appears as the asymptotic expansion of $\mathcal{S}^\theta(\varphi).$

Before we describe resurgence, we need a result which will allow us to move between shifted Borel--Laplace sums. The result below follows from a straightforward application of results for the unshifted Borel and Laplace transforms (e.g. Lemma 3, Section 3.2, and Theorem 2, Section 3.3, of \cite{Balser}) translated to our setting.

\begin{lemma} \label{lem: borel laplace shifts} Let $\varphi(z)$ be a Gevrey-1 formal power series. Suppose $\kappa, \kappa' \in \frac{1}{2}\mathbb{Z}$ and $\theta \in \mathbb{R}$ are such that $\mathcal{S}_\kappa^\theta(\varphi)$ and $\mathcal{S}_{\kappa'}^\theta(\varphi)$ are analytic for $\{\Re(ze^{-i\theta})>0\}$. Then we have
	\[\mathcal{S}_\kappa^\theta(\varphi) = \mathcal{S}_{\kappa'}^\theta(\varphi).
	\]
\end{lemma}
Thus, we drop the subscript $\kappa$ when appropriate. We will also use the fact that taking Borel-Laplace sums is linear (Theorem 2, Section 3.3 of \cite{Balser}). 

Returning to Gevrey-1 formal power series, we will be interested in when Borel transforms of such series admit an analytic continuation beyond their original neighborhood of convergence. The functions which will be the subject of later sections will primarily have particularly nice analytic continuations and will be of exponential type 0 in the directions we sum. 

 We will say an analytic function $\psi(u)$ defined on the unit disk $|u| < 1$ is \textit{resurgent} if it admits endless analytic continuation (see \cite{ecalle} for a precise definition). 
	 Further, if $\psi(u)$ has only simple poles and logarithmic branch points, we say it is \textit{simple resurgent}. 

In the cases we will entertain for the remainder of the paper, the Borel transforms of our formal power series $\varphi(z)$ will be simple resurgent and will admit only simple poles which are located along finitely many rays, called Stokes rays. These rays divide the $u$-plane into finitely many sectors, and to each sector one associates a directional Laplace transform by the above definition. 

If $\psi$ is simple resurgent with a simple pole at $\omega$, it has the local expansion

\[\psi(u) = -\frac{S_\omega}{2\pi i(u-\omega)} + F(u),
\]
where $F$ is holomorphic in a neighborhood of $\omega$. We call the constant $S_\omega$ the Stokes constant corresponding to $\omega$.

Now we investigate what happens as we vary the argument of the Laplace transform to cross a Stokes ray. Suppose $\Omega$ is the set of poles of $\psi$ and set $\theta := \arg(\omega)$ where $\omega \in \Omega$. If $\Omega_\theta$ is the subset of all poles lying along this ray, then there exists $\varepsilon>0$ such that for any $(\theta^-, \theta^+)$ satisfying $0<\theta^+ - \theta < \varepsilon$ and $0< \theta - \theta^- < \varepsilon$, we have
\[\mathcal{L}_\kappa^{\theta^+}(\psi)(z)-\mathcal{L}_\kappa^{\theta^-}(\psi)(z) = z^{-\kappa}\sum_{\omega \in \Omega_\theta} S_\omega \omega^{\kappa-1} e^{-\omega/z} =: \text{disc}_\kappa^\theta(\psi(u)),
\]
which we will call the discontinuity of $\psi$. 
In order to define a ``canonical" Borel-Laplace sum of a divergent series, one must determine how resum independent of a choice of angle $\theta$. This motivates the construction of the median resummation of a divergent series $\varphi$. 
In general, median resummation is a more complex procedure than averaging, but in the case where $\mathcal{S}^{\theta}(\varphi)$ is analytic in a sector around $e^{i\theta}\mathbb{R}_+$, it takes the form 
\[\mathcal{S}_{med}^\theta(\varphi):=\frac{\mathcal{S}^{\theta^+}(\varphi) + \mathcal{S}^{\theta^-}(\varphi)}{2}.\] 
This function is analytic in the sector with bisecting direction $\theta$ and opening $\pi$. In the case where $\mathcal{B}(\varphi)$ admits a single Stokes ray, this defines a unique function corresponding to $\varphi$. We may also write
\[\mathcal{S}_{med}^\theta(\varphi)=\begin{cases} \mathcal{S}^{\theta^-}(\varphi) + \frac{1}{2}\text{disc}_{\kappa}^\theta(\mathcal{B}_{\kappa}(\varphi)), & \Re(ze^{-i\theta^-}), \Re(ze^{-i\theta})>0, \\
\mathcal{S}^{\theta^+}(\varphi) - \frac{1}{2}\text{disc}_{\kappa}^\theta(\mathcal{B}_{\kappa}(\varphi)), & \Re(ze^{-i\theta^+}), \Re(ze^{-i\theta})>0.
\end{cases}
\] 
This decomposition will be essential to our later realization of quantum modularity. 



\subsection{Eichler Integrals of Cusp Forms of Integer and Half-Integer Weight} \label{subsec: modular}

We begin with a brief recollection of the required definitions and properties of modular forms. 

To start, for functions $f \colon \mathbb{H} \to \mathbb{C}$, we recall the action of the Petersson slash operator of weight $k \in \frac{1}{2}\mathbb{Z}$ associated to $\gamma = \begin{psmallmatrix} a & b \\ c& d \end{psmallmatrix} \in \Gamma_0(N)$ (where $4 \mid N$ if $k \in \frac{1}{2} + \mathbb{Z}$):
\begin{equation}\label{eqn: slash}
f|_k\gamma(z):= \begin{cases} (cz+d)^{-k}f\left(\frac{az+b}{cz+d}\right) & \text{for } k \in \mathbb{Z},\\
 \varepsilon_d^{2k}\left(\frac{c}{d}\right) (cz+d)^{-k} 	f\left(\frac{az+b}{cz+d}\right) & \text{for } k \in \frac{1}{2}+ \mathbb{Z}.
 \end{cases}
\end{equation}
Here, $\left(\frac{\cdot}{\cdot}\right)$ denotes the Kronecker symbol and for odd $d$,
\[\varepsilon_d:= \begin{cases}
 1 & \text{if } d \equiv 1 \pmod{4},\\
 i & \text{if } d \equiv 3 \pmod{4}.
 \end{cases}
\]
We will consider modular forms on congruence subgroups
\[\Gamma_0(N) :=\left\{ \begin{pmatrix} a & b \\ c& d \end{pmatrix} \in \text{SL}_2(\mathbb{Z}) : c \equiv 0 \pmod{N}\right\}.
\]
We are now able to define the spaces of modular forms we will consider.

\begin{definition}
	Let $k \in \frac{1}{2}\mathbb{N}$ and $N \in \mathbb{N}$ (where $4 \mid N$ when $k$ is a half-integer). We say a holomorphic function $f\colon \mathbb{H} \to \mathbb{C}$ is a cusp form of weight $k$ for $\Gamma_0(N)$ if:
	\begin{enumerate}
	\item For all $\gamma = \begin{psmallmatrix} a & b \\ c& d \end{psmallmatrix} \in \Gamma_0(N)$, we have $f(z) = f|_{k}\gamma(z)$ for all $z \in \mathbb{H}$;
	\item As $z$ approaches any cusp of $\Gamma_0(N)$, $f$ decays exponentially.	
	\end{enumerate}
\end{definition}
Such functions necessarily admit Fourier expansions at infinity without constant term. We will often define our function $f$ by such an expansion. 

We will not only consider cusp forms themselves, but also their Eichler integrals. Miraculously, in the case where $k \in \mathbb{N}$, there is a canonical differential operator $D^{k-1}$, where $D = \frac{1}{2\pi i }\frac{\partial}{\partial z}$,  which sends modular forms of weight $2-k$ to modular forms of weight $k$, and so it is natural to consider the $(k-1)$-fold primitive of a cusp form $f$, defined by \[\tilde{f}(z) := \frac{(-2\pi i z)^{k-1}}{\Gamma(k-1)}\int_z^{i\infty} f(\tau) (\tau-z)^{k-2}\;d\tau,\] for $z$ in $\mathbb{H}$. If $f$ has a Fourier expansion
$f(z) = \sum_{n=1}^\infty a_n q^n$, where $q =e^{2\pi i z}$,  then $\tilde{f}$ has the Fourier expansion 
\begin{equation}\tilde{f}(z) = \sum_{n=1}^\infty a_n n^{1-k} q^n. \label{eqn: eichler} \end{equation} 

Such a function $\tilde{f}$ is nearly a modular form of weight $2-k$; its obstruction to modularity is a polynomial of degree at most $k-2$. 
For $\gamma \in \Gamma_0(N)$, one can write the difference
\[\tilde{f} (z) - \tilde{f}|_{2-k}\gamma(z)=\frac{(-2\pi i z)^{k-1}}{\Gamma(k-1)}\int_{\gamma^{-1}(i\infty)}^{i\infty} f(\tau) (\tau - z)^{k-2}\;d\tau.\]

The half-integer analogue was first considered in \cite{lawrence-zagier}. In the case where $k \in \frac{1}{2}+\mathbb{N}_0$, the integral definition is no longer sensible, but the definition on the level of Fourier expansions is still reasonable. The obstruction to modularity is much more interesting in this case. 

It was first shown for a family of examples of false theta functions \cite{lawrence-zagier}, and then extended to the general case in \cite{bringmann-rolen}, that if one considers the limiting values of $\tilde{f}$ as  $z$ approaches  $x \in \mathbb{Q}$ along a geodesic, then the resulting function (still denoted by $\tilde{f}$) from $\mathbb{Q}$ to $\mathbb{C}$ forms a quantum modular form.  Namely, the obstructions to modularity
\[h_\gamma(x) := \tilde{f}(x) - \chi_{-4}(\gamma)\tilde{f}|_{2-k}(x),
\]
extend to real analytic functions on $\mathbb{R}\backslash \{\gamma^{-1}(i\infty)\}$ for all $\gamma \in \Gamma_0(N)$. We defined $\chi_{-4}(n)$ in Section 1.1.

This result was proven by considering a companion function $f^*$ which acts as a non-holomorphic Eichler integral on the lower half-plane. This function has the same asymptotics as $\tilde{f}$ as one approaches the same rational number. This function can be shown to have the obstruction to modularity 
\[f^*(z) - \chi_{-4}(\gamma) f^*|_{2-k}\gamma(z) = \frac{(-2\pi i z)^{k-1}}{\Gamma(k-1)}\int_{\gamma^{-1}(i\infty)}^{i\infty} f(\tau) (\tau - z)^{k-2}\;d\tau
\]
in the lower half-plane, which $\tilde{f}$ inherits on $\mathbb{Q}$. We recall its explicit form here because this function appears again in our results. 

To conclude this section, we recall the relevant functional equations for $L$-functions associated to modular forms. 
Suppose that $f$ is a cusp form of weight $k \in \frac{1}{2}\mathbb{N}$ on $\Gamma_0(N)$ for $N \in \mathbb{N}$ (where $4\mid N$ in the half-integer weight case) which has Fourier expansion $f(z) = \sum_{n=1}^\infty a_n q^n$.
For $-d/c \in \mathbb{Q}$, define $\zeta_c^{-d} := e^{2\pi i (-d/c)}$ and
\[L(f, \zeta_c^{-d};\rho) := \sum_{m = 1}^\infty \frac{a_m \zeta_c^{-dm}}{m^\rho}, \quad \quad \Re(\rho)\gg 0.
\]
Using the Mellin transform of $f$, we define the completed $L$-function by
\[\Lambda(f, \zeta_c^{-d}; \rho) := \left( \frac{c}{2\pi}\right)^\rho \Gamma(\rho) L(f, \zeta_c^{-d}; \rho) = c^\rho \int_0^\infty f(it-d/c) t^{\rho-1}\;dt.
\]
 
 \begin{prop} Suppose $f$ is as above. Then the following are true. 
 \begin{enumerate}
 \item 	If $k \in \mathbb{N}$ and $\gamma = \begin{psmallmatrix} a & b \\ c & d\end{psmallmatrix} \in \Gamma_0(N)$, then $\Lambda(f, \zeta_c^{-d}; \rho)$ analytically continues to $\mathbb{C}$ and satisfies
  \[\Lambda(f, \zeta_c^{-d}; \rho) = i^{-k} \Lambda(f, \zeta_c^a;  k-\rho).
	\]
 \item If $k \in \frac{1}{2} + \mathbb{N}$ and $\gamma = \begin{psmallmatrix} a & b \\ c & d\end{psmallmatrix} \in \Gamma_0(N)$, then then $\Lambda(f, \zeta_c^{-d}; \rho)$ analytically continues to $\mathbb{C}$ and satisfies
 \[\Lambda(f, \zeta_c^{-d}; \rho) = i^{-k} \varepsilon_d^{2k} \left( \frac{c}{d}\right) \Lambda(f, \zeta_c^a; k-\rho).
	\]
 \end{enumerate}
 \end{prop}
 
 \begin{proof} Both of the claims are routinely proven, but we offer a proof of the second claim as it appears less frequently in the literature. 
We first write 
\[c^\rho \int_0^\infty f(it-d/c) t^{\rho-1}\;dt = c^\rho \int_0^{1/c} f(it-d/c) t^{\rho-1}\;dt + c^\rho \int_{1/c}^{\infty} f(it-d/c) t^{\rho-1}\;dt.
\]	
Under the change of variables $t \mapsto 1/c^2t$, the first integral becomes
\begin{align*}c^\rho \int_0^{1/c} f(it-d/c) t^{\rho-1}\;dt &= c^\rho \int_\infty^{1/c} f(-d/c+i/c^2t)(c^2)^{1-\rho}t^{1-\rho} \frac{-1}{c^2t^2}\;dt \\
&= c^{-\rho} \int_{1/c}^\infty f(-d/c+i/c^2t)t^{-\rho-1} \;dt
\end{align*}
Then recognizing that if $\gamma = \begin{psmallmatrix} a & b \\ c & d\end{psmallmatrix}$, then $\gamma(-d/c+i/c^2t) = a/c + it$, we have
\begin{align*}
c^{-\rho} \int_{1/c}^\infty f(-d/c+i/c^2t)t^{-\rho-1} \;dt 
&= c^{-\rho} \int_{1/c}^\infty f(-d/c+i/c^2t)t^{-\rho-1}\;dt\\ 
&= c^{-\rho} \int_{1/c}^\infty  f|_{k}\gamma(-d/c+i/c^2t)t^{-\rho-1} \;dt \\ 
&= i^{-k}  \varepsilon_d^{2k} \left( \frac{c}{d}\right)  c^{k-\rho} \int_{1/c}^\infty f(a/c+it) t^{k-\rho-1}\;dt.
\end{align*}	
Above, we used that $\gamma \in \Gamma_0(N)$. The above calculation gives a presentation for $\Lambda(f, \zeta_{c}^{-d};\rho)$ which immediately implies both the desired analytic continuation to $\mathbb{C}$ and the stated functional equation. 
\end{proof}

\subsection{Maass Waveforms} \label{subsec: maass}

We now recall the basic definitions and properties of Maass waveforms for $\text{SL}_2(\mathbb{Z})$. We also refer the reader to \cite{bump, iwaniec} for a more thorough introduction to the theory. 

We begin with a smooth function $u \colon \mathbb{H} \to \mathbb{C}$ satisfying the following.
\begin{enumerate}
	\item For all $\gamma = \begin{psmallmatrix} a & b \\ c& d \end{psmallmatrix} \in \Gamma_0(N)$ and for all $z \in \mathbb{H}$, we have $u(\gamma z)=u(z)$;
	\item If $\Delta := -y^2\left( \frac{\partial^2}{\partial x^2}+\frac{\partial^2}{\partial y^2}\right)$ is the hyperbolic Laplacian, then $\Delta u = \lambda u$ for some $\lambda \in \mathbb{C}$.
\end{enumerate}
Such a function necessarily admits a Fourier--Whittaker expansion. If the zeroth Fourier--Whittaker coefficient vanishes and $u$ is square-integrable on a fundamental domain of $\text{SL}_2(\mathbb{Z})$ with respect to the hyperbolic measure, then we say that $u$ is a Maass waveform. 

We call $\lambda$ the eigenvalue of $u$ and if $\lambda = s(1-s)$, then $s$ is called its spectral parameter.
Such a function $u(z)$ has an expansion of the form 
\begin{equation}\label{eqn: maass} u(z) = \sqrt{y}\sum_{n \in \mathbb{Z}_{\neq 0}} A_n K_{s-1/2}(2\pi |n| y) e^{2\pi i nx}, \quad \quad (z = x+iy \in \mathbb{H}),
\end{equation}
where $K_\nu(x)$ is the $K$-Bessel function of order $\nu$.

In the central result of \cite{lewis-zagier}, Lewis and Zagier gave a correspondence between Maass waveforms and three different classes of functions. 
While all show up in our results, we will only recall two of them here. Let $u(z)$ be a Maass waveform with spectral parameter $s$, where $\Re(s)>0$.
Then associated to $u$ are a pair of Dirichlet $L$-series $L_\varepsilon(\rho)$, $\varepsilon = \{0, 1\}$, convergent in some
right half-plane, such that the functions $\Lambda_{\varepsilon}(\rho) :=  \gamma_s(\rho+\varepsilon)L_\varepsilon(\rho)$, where
\[\gamma_s(\rho) := \frac{1}{4\pi^\rho}\Gamma\left(\frac{\rho-s+1/2}{2}\right) \Gamma\left(\frac{\rho+s-1/2}{2}\right),
\]
are entire functions of finite order and satisfy 
\[\Lambda_\varepsilon(1-\rho) = (-1)^\varepsilon \Lambda_\varepsilon(\rho).
\]
From the coefficients in the expansion of $u$ in (\ref{eqn: maass}), their Dirichlet series take the form
\[L_\varepsilon(\rho) := \sum_{m=1}^\infty \frac{A_m + (-1)^\varepsilon A_{-m}}{m^\rho}, \quad \quad (\varepsilon \in \{0,1\}). \]
Moreoever, there is also a corresponding holomorphic periodic function $f\colon \mathbb{C}\backslash \mathbb{R} \to \mathbb{C}$, defined by
\[f(z) := \begin{cases} \sum_{n>0}^\infty A_n n^{s-1/2} q^n, & \text{if } z \in \mathbb{H},\\
 	-\sum_{n<0}^\infty A_{n} |n|^{s-1/2} q^n, & \text{if } z \in \mathbb{H}^-,
 \end{cases}
\]
where $\mathbb{H}^-$ denotes the lower half-plane. 
From this point on, we will abuse notation and use $f$ to refer to the restriction of $f$ to the upper half-plane. Throughout, we will prove results for $f$ in the upper half-plane, but the corresponding results for $f$ in the lower half-plane are completely analogous, and in fact $f$ on $\mathbb{C} \backslash\mathbb{R}$ is really a single holomorphic quantum modular form. 

We will also consider one non-cuspidal case in our work. We will define the non-holomorphic Eisenstein series with spectral parameter $s$ by

\begin{equation}\begin{aligned} \label{eqn: maass eisenstein}
E_s(z) &:= \zeta(2s)y^s + \frac{\pi^{1/2}\Gamma(s-\frac{1}{2})}{\Gamma(s)}\zeta(2s-1)y^{1-s} \\&+ \frac{4\pi^s}{\Gamma(s)}\sqrt{y}\sum_{n=1}^\infty n^{1/2-s}\sigma_{2s-1}(n) K_{s-1/2}(2\pi ny) \cos(2\pi nx), 
\end{aligned}
\end{equation}
where $\sigma_\nu(n) := \sum_{d\mid n} d^\nu$ is the $\nu$-th sum of divisors function. We will specifically analyze the case where $s=1/2$. As is standard, we will write $d(n)$ for $\sigma_0(n).$ For the associated holomorphic periodic function in this case, we pick the normalization 
\[f(z):= 1-4\sum_{n=1}^\infty d(n)q^n.
\]

Above, we gave the functional equations for the pair of $L$-functions associated to a Maass waveform. Here, we state the analogous result for our Eisenstein series of interest. 
The corresponding $L$-function is proportional to
\[\zeta^2(\rho) = \sum_{m=1}^\infty \frac{d(m)}{m^\rho}, \quad \quad \Re(\rho)>1,
\]
where $\zeta(\rho)$ is the Riemann zeta function. The completion of the function $\zeta^2(\rho)$ has meromorphic continuation to $\mathbb{C}$ with a pole at $s=1$ and satisfies the functional equation 
\[\pi^{-\rho} \Gamma\left(\frac{\rho}{2}\right)^2 \zeta^2(\rho) = \pi^{-(1-\rho)} \Gamma\left(\frac{1-\rho}{2}\right)^2 \zeta ^2(1-\rho).
\]

\section{Intermediate results}\label{sec: intermediate}

\subsection{Asymptotics} \label{subsec: asymptotics}

Throughout this paper, we will regularly leverage the relationship between the coefficients of an asymptotic expansion of an exponential series and the analytic continuation of an associated Dirichlet series. For a more thorough discussion, we direct the reader to \cite{zagier-appendix}, but for most of our applications, the following result will suffice:
\begin{prop}[Proposition 2, \cite{zagier-appendix}]
Let $L(\rho)$ be a Dirichlet series convergent in some right half-plane
\[L(\rho) = \sum_{m=1}^\infty \frac{a_m}{m^\rho}, \quad \quad \Re(\rho) \gg 0.
\]
Suppose that the function $\phi(t) = \sum_{m=1}^\infty a_me^{-mt}$ has an asymptotic expansion of the form 
\[\phi(t) \sim \sum_{n=0}^\infty c_n t^n, \quad \quad (t \to 0).
\]
Then $L(\rho)$ has an analytic continuation to all $\rho$, and its values at non-positive integers are given by
\[ L(-n) = (-1)^n n! c_n.
\]
Conversely, if one assumes $L(\rho)$ analytically continues to $\mathbb{C}$, then $\phi(t)$ admits an asymptotic expansion with coefficients given by 
\[c_n  = (-1)^n \frac{L(-n)}{n!}.
\]
\end{prop}

Now we compute the Borel transforms of the asymptotic series attached to cusp forms, their Eichler integrals, and periodic functions of Maass forms. In what follows, we assume $\gamma = \begin{psmallmatrix} a & b \\ c & d\end{psmallmatrix} \in \Gamma_0(N)$, and that the asymptotics are taken along geodesics in the upper half-plane. 

When $g$ is of \emph{Type 1} or \emph{2}, the above proposition tells us that as $z \to -d/c$ along a geodesic, 

\[\tilde{g}(z) \sim \varphi_{-d/c}(z) := \sum_{n=0}^\infty \frac{L(g, -d/c; -n)}{\Gamma(n+1)} (2\pi i)^n (z+d/c)^n.
\]
Using the fact that $\Lambda(g, -d/c; \rho)$ is entire and $\Gamma(\rho)$ has poles at non-positive integers, $L(g, -d/c; \rho)$ is forced to have zeros at these locations. Therefore each term equals zero. This is to be expected since $g$ is a cusp form. 

When $g$ is of \emph{Type 3}, as $z \to -d/c$, we have
\[g(z) \sim \varphi_{-d/c}(z) := \sum_{n=0}^\infty \frac{L(f, -d/c; -n-1+k)}{\Gamma(n+1)} (2\pi i)^n (z+d/c)^n,
\]
where $f$ is the cusp form associated to $g$. When $-n-1+k$ is a non-positive integer, we again have that $L(f, -d/c; -n-1+k)=0$. Then only the first $k-2$ terms are nonzero. Using the functional equation of $L(f, -d/c; \rho)$, this equals
\[\sum_{n=0}^{k-2} i^{-k} \left(\frac{c}{2\pi}\right)^{2n+2-k} \frac{L(f, a/c; n+1)}{\Gamma(-n-1+k)} (2\pi i)^n (z+d/c)^n. 
\]

We now consider when $g$ is of \emph{Type 4}. Again we have that as $z \to -d/c$, 

\[g(z) \sim \varphi_{-d/c}(z) := \sum_{n=0}^\infty \frac{L(f, -d/c; -n-1+k)}{\Gamma(n+1)} (2\pi i)^n (z+d/c)^n,
\]
where $f$ is the cusp form associated to $g$. By the functional equation of $L(f, -d/c;\rho)$, we have that this equals
\[\sum_{n=0}^\infty i^{-k} \varepsilon_d^{2k} \left( \frac{c}{d}\right) \left(\frac{c}{2\pi}\right)^{2n+2-k} \frac{L(f, a/c; n+1)}{\Gamma(-n-1+k)} (2\pi i)^n (z+d/c)^n.
\]
 Using that 
\[\frac{1}{\Gamma(-n-1+k)} = \frac{\sin(\pi(n+2-k))}{\pi} \Gamma(n+2-k),
\]
and that $\sin(m\pi + \pi/2)= (-1)^m$, after rearranging we have
\[\frac{1}{\pi i}  \frac{\varepsilon_d^{2k} \left( \frac{c}{d}\right) }{c^{2-k}}\sum_{n=0}^\infty \Gamma(n+2-k) \frac{L(f, a/c; n+1)}{(2\pi i /c^2)^{n+2-k}} (z+d/c)^n. 
\]
The fact that the above asymptotic series is Gevrey-1 is clear. 
Now we split the sum along $n=\ell$ for $\ell$ sufficiently large so that we can both leverage our Dirichlet series representation for $L(f, a/c; \rho)$ and ensure the convergence of a Laplace transform we will apply later. We write
\begin{align*}\frac{1}{\pi i}  \frac{\varepsilon_d^{2k} \left( \frac{c}{d}\right) }{c^{2-k}} &\sum_{n=0}^{\ell-1}\Gamma(n+2-k) \frac{L(f, a/c; n+1)}{(2\pi i /c^2)^{n+2-k}} u^n \\&+ \frac{1}{\pi i} \frac{\varepsilon_d^{2k} \left( \frac{c}{d}\right)}{c^{2-k}}\sum_{n=\ell}^{\infty} \Gamma(n+2-k)\frac{L(f, a/c; n+1)}{(2\pi i /c^2)^{n+2-k}} u^n.
\end{align*}

We will deal with the finitely many leading terms separately.  For future use, we will call the first sum $\widehat{\varphi}_{-d/c}^{\ell}(z)$ second sum $\varphi_{-d/c}^\ell(z)$.  For now, taking the Borel transform with shift $2-k$ with respect to $z+d/c$ of $\varphi_{-d/c}^\ell$ gives
\[\frac{1}{\pi i}  \frac{\varepsilon_d^{2k} \left( \frac{c}{d}\right)}{c^{2-k}}\sum_{n=\ell}^\infty \frac{L(f, a/c; n+1)}{(2\pi i /c^2)^{n+2-k}} u^n.
\]
This can be rewritten as 
\begin{align*}
\frac{1}{\pi i}  \frac{\varepsilon_d^{2k} \left( \frac{c}{d}\right) }{c^{2-k}}\sum_{n=\ell}^{\infty} \sum_{m=1}^\infty \frac{a_m m^{1-k}\zeta_c^{am}}{\left(\frac{2\pi i m}{c^2}\right)^{n+2-k}} u^n 
= \frac{1}{\pi i}  \frac{\varepsilon_d^{2k} \left( \frac{c}{d}\right) }{c^{2-k}} \sum_{m=1}^\infty \frac{a_m m^{1-k}\zeta_c^{am}}{\left(\frac{2\pi i m}{c^2}\right)^{2-k}}\sum_{n=\ell}^{\infty}  \left(\frac{u}{\left(\frac{2\pi i m}{c^2}\right)}\right)^n
\end{align*}
where we used absolute convergence to switch the order of summation. 
From here, summing the geometric series gives 
\[\frac{-1}{\pi i}  \frac{\varepsilon_d^{2k} \left( \frac{c}{d}\right) }{c^{2-k}}  \sum_{m=1}^\infty \frac{a_m m^{1-k}\zeta_c^{am} }{u-\left(\frac{2\pi i m}{c^2}\right)}\frac{u^{\ell}}{\left(\frac{2\pi i m}{c^2}\right)^{\ell+1-k}},
\]
which has simple poles at $\omega_m = \frac{2\pi i m}{c^2}$ with residues $-2  \frac{\varepsilon_d^{2k} \left( \frac{c}{d}\right)}{c^{2-k}}  a_m m^{1-k}\zeta_c^{am}\omega_m^{k-1}$. The sum $\widehat{\varphi}_{-d/c}^{\ell}(z)$ will be resummed separately.

Turning to $g$ of \emph{Type 5}, we appeal to 
the following result which is a corollary of Proposition 4.1 of \cite{fantini-rella}:

\begin{prop} \label{prop: maass discont}
	Let $g$ be of Type 5 and let $\varphi_0(z)$ denote the asymptotic expansion of $g(z)$ as $z \to 0^+$ along a geodesic. Then we have
	\[
	\varphi_0(z)= \frac{1}{\pi i}\sum_{n =0}^\infty \Gamma(n+1) \sum_{m \neq 0}\frac{A_m}{(2\pi i m)^{n+1}}z^n.
	\]
\end{prop}
Taking a Borel transform of $\varphi(z)$, we have 
\begin{align*}\mathcal{B}_1(\varphi_0(z))= \frac{1}{\pi i}\sum_{n =0}^\infty  \sum_{m \neq 0}\frac{A_m}{(2\pi i m)^{n+1}}u^n=  \frac{1}{\pi i} \sum_{m \neq 0}\frac{A_m}{2\pi i m} \sum_{n =0}^\infty  \frac{u^n}{(2\pi i m)^{n}}=  \frac{-1}{\pi i} \sum_{m \neq 0} \frac{A_m}{u-2\pi i m}.	\end{align*}
This ultimately tells us that $\mathcal{B}_1(\varphi_0(z))$ has a meromorphic continuation to $\mathbb{C}$ with simple poles at $2\pi i m$ for $m \in \mathbb{Z}_{\neq 0}$ with residues $-2A_m$. 

In the case where $g$ is of \emph{Type 6}, the asymptotics require slightly more care. 
\begin{prop} \label{prop: eisenstein asymptotic}
	Let $g$ be of Type 6 and let $\varphi^*_0(z)$ denote the asymptotic expansion of $g(z)$ as $z \to 0^+$ along a geodesic. Then we have
		\[
	\varphi^*_0(z)= -8\frac{\log(-2\pi i z)-\gamma}{2\pi i z} -\frac{8}{\pi i} \sum_{\substack{n=1 \\ n \text{ odd}}}^\infty \Gamma(n+1) \sum_{m=1}^\infty \frac{d(m)}{(2\pi im)^{n+1}} z^n.
	\]
\end{prop}

\begin{proof}
One considers the Mellin transform of $\sum_{n=1}^\infty d(n) q^n$ and moves the contour over the pole at $s=1$ to compute its contribution. We have 
\[g(z) = 1-4\sum_{n=1}^\infty d(n) q^n \sim -8\frac{\log(-2\pi i z)-\gamma}{2\pi i z} - 4 \sum_{n=1}^\infty \frac{\zeta^2(-n)}{n!}(2\pi i z)^n,
\]
where $\gamma$ is the Euler-Mascheroni constant. From this point, we analyze the latter sum, which we call $\varphi_0(z)$. Using the functional equation for $\zeta^2$ and identities for the $\Gamma$-function, as formal power series we have
\[- 4 \sum_{n=1}^\infty \frac{\zeta^2(-n)}{n!}(2\pi i z)^n = -4\sum_{\substack{n=1 \\ n \text{ odd}}}^\infty \Gamma(n+1) \frac{i^n}{2^n\pi^{n+2}} \zeta^2(n+1) z^n. 
\]
Now we are in the range where the Dirichlet series representation of $\zeta^2$ converges, and so we may expand and rearrange to arrive at

\[\frac{-8}{\pi i} \sum_{\substack{n=1 \\ n \text{ odd}}}^\infty \Gamma(n+1) \sum_{m=1}^\infty \frac{d(m)}{(2\pi im)^{n+1}} z^n. \qedhere
\]\end{proof}
Now considering the Borel transform of $\varphi_0$ and evaluating the geometric series, we have that
\begin{align*}\mathcal{B}_1(\varphi_0(z)) &= \frac{-8}{\pi i} \sum_{\substack{n=1 \\ n \text{ odd}}}^\infty \sum_{m=1}^\infty \frac{d(m)}{(2\pi im)^{n+1}} u^n 
=  \frac{-8}{\pi i} \sum_{m=1}^\infty \frac{d(m)u}{(2\pi im)^{2} -u^2}.
\end{align*}
Using the same strategy as in the proof of Proposition 4.1 of \cite{fantini-rella}, this becomes
\[\frac{4}{\pi i} \sum_{m\neq 0} \frac{d(|m|)}{z-2\pi i m}.
\]
Then $\mathcal{B}_1(\varphi_0(z))$ has meromorphic continuation to $\mathbb{C}$ with simple poles at $2\pi i m$ for $m \in \mathbb{Z}_{\neq 0}$ with residues $8d(|m|)$. 

\subsection{Median Resummation} \label{subsec: median resum}

In this section, for our functions $g$ of interest, suppose that $g(z) \sim \varphi_{-d/c}(z)$ as $z$ approaches $-d/c \in \mathbb{Q}$. For appropriate $-d/c$, we will prove decompositions of the form 
\begin{equation}\label{eqn: decomp integral}
g(z)
=  \mathcal{S}^{\pi/2}_{med}(\varphi_{-d/c})(z) 
+  \int_{\Gamma}e^{-u/(z+d/c)}\widehat{g}_{-d/c}^{odd}(u)\frac{du}{\sqrt{u}},
\end{equation}
where $\widehat{g}_{-d/c}^{odd}(u)$ is a function related to $g$ and the contour $\Gamma$ is displayed in Figure \ref{fig: contours}. In the first cases we will consider, we will show that the second term in the sum vanishes, giving a median resummation result for these functions. In the latter cases, the situation is more interesting, as the second term in the sum will contribute. The general approach for the results in this section was inspired by similar work in \cite{gukov-marino-putrov, andersen-mistegard} for false theta functions of weight $1/2$.

\begin{prop}\label{prop: med resum} Suppose $g$ is of Type 4 with respect to $\Gamma_0(N)$ or of Type 5, and suppose $\gamma = \begin{psmallmatrix} a & b \\ c & d\end{psmallmatrix} \in \Gamma_0(N)$. Then $g= \mathcal{S}_{med}^{\pi/2}(\varphi_{-d/c})$.	
\end{prop}

\begin{proof}
Let $g(z)= \sum_{n=0}^\infty b_n q^n$ and suppose that $g(z) \sim \varphi_{-d/c}(z)$ as $z$ approaches $-d/c \in \mathbb{Q}$.
Now define $\widehat{g}_{-d/c}(u) := \frac{1}{\sqrt{\pi}}\sum_{n =0}^\infty b_n \zeta_c^{-dn}  e^{-\sqrt{8\pi i  nu}}$. We will relate $g$ and $\widehat{g}_{-d/c}$ through Borel resummation.

We will prove our result for $z+d/c$ along the half-line $i\mathbb{R}_+$ and deduce the equalities in Propositions \ref{prop: med resum} on $\mathbb{H}$ by holomorphicity. Using Gaussian integration, we have that for $\varepsilon >0$,
\[e^{2\pi i n (z+d/c)} = \frac{1}{\sqrt{\pi}\sqrt{8\pi i (z+d/c)}} \int_{i\mathbb{R}+\varepsilon}e^{-\sqrt{n}\eta} e^{-\eta^2/8\pi i (z+d/c)}\;d\eta.
\]
Then we may write 
\begin{align*}
\sum_{n=1}^\infty b^n q^n &= \sum_{n = 1}^\infty b^n \zeta_c^{-dn} e^{2\pi i (z+d/c)} \\&=  \frac{1}{\sqrt{\pi}\sqrt{8\pi i (z+d/c)}}\int_{i\mathbb{R}+\varepsilon}\sum_{n \geq 1} b_n\zeta_c^{-dn} e^{-\sqrt{n}\eta} e^{-\eta^2/8\pi i (z+d/c)}\;d\eta.
\end{align*}
Under the change of variables $\eta^2/8\pi i =u$, this becomes
\[\sum_{n = 1}^\infty b_n q^n =  \frac{1}{2\sqrt{z+d/c}}\int_{\Gamma}e^{-u/(z+d/c)} \sum_{n = 1}^\infty b_n\zeta_c^{-dn} e^{-\sqrt{8\pi i nu}} \;\frac{du}{\sqrt{\pi u}},
\]
for the contour $\Gamma$ displayed in Figure \ref{fig: contours}.
\begin{figure}
	\begin{tikzpicture}
  \draw[-] (-2,0) -- (2,0); 
  \draw[-] (0,-1) -- (0,3);  
  \draw[thick, black, 
        postaction={decorate},
        decoration={
          markings,
          mark=at position 0.215 with {\arrow{>}},
          mark=at position 0.8 with {\arrow{>}}
        }]
    plot[domain=-.58:.58, samples=200] 
    (\x, {20*\x*\x*\x*\x + 2*\x*\x- .25});
\end{tikzpicture}
	\caption{The contour $\Gamma$.}
  \label{fig: contours}
\end{figure}
Expanding $\widehat{g}_{-d/c}(u)$ as a Puiseux series supported on half-integer powers, we have 
\begin{equation}\label{egn: g hat asymptotic} 
\widehat{g}_{-d/c}(u) = \frac{1}{\sqrt{\pi}}\sum_{n \geq 1} b_n\zeta_c^{-dn} e^{-\sqrt{8\pi i nu}} \sim \frac{1}{\sqrt{\pi }} \sum_{r=0}^\infty \frac{(-1)^r L(g, \zeta_c^{-d};-r/2)}{\Gamma(r+1)}(8\pi iu)^{r/2}.
\end{equation}
We split this series based on whether $r$ is even or odd. We write
\[
\widehat{g}_{-d/c}^{even}(u) := \frac{1}{\sqrt{\pi }} \sum_{n=0}^\infty \frac{ L(g, \zeta_c^{-d};-n)}{\Gamma(2n+1)}(8\pi iu)^{n};
\]
\[
\frac{1}{\sqrt{u}}\widehat{g}_{-d/c}^{odd}(u) := -\frac{1}{\sqrt{\pi }} \sum_{n=0}^\infty \frac{ L(g, \zeta_c^{-d};-n-1/2)}{\Gamma(2n+2)}(8\pi i)^{n+1/2}u^{n}.
\]
Since in both cases, we assume a functional equation of degree 2 for some $L$-function related to $L(g, \zeta_c^{-d}; \rho)$ by a shift, both of the above power series converge absolutely in some neighborhood of $u=0$. 
Moreover, since $\Gamma(n+1)\Gamma(n+1/2) = \frac{\sqrt{\pi}}{4^n} \Gamma(2n+1)$, we see that 
\[
\mathcal{B}_{1/2}(\varphi_{-d/c}(z)) = \frac{1}{\sqrt{\pi }} \sum_{n=0}^\infty \frac{ L(g, \zeta_c^{-d};-n)}{\Gamma(2n+1)}(8\pi iu)^{n} = \widehat{g}_{-d/c}^{even}(u).
\]
Then the above allows us to write
\begin{equation*}
\sum_{n = 1}^\infty b_n q^n =  \frac{1}{2\sqrt{z+d/c}} \left( \int_{\Gamma}e^{-u/(z+d/c)} \mathcal{B}_{1/2}(\varphi_{-d/c}(z)) \frac{du}{\sqrt{\pi u}} \ + \ \int_{\Gamma}e^{-u/(z+d/c)}\widehat{g}_{-d/c}^{odd}(u)\frac{du}{\sqrt{u}}\right).
\end{equation*}
By deforming contours, we have
\[
\int_\Gamma e^{-u/(z+d/c)}\mathcal{B}_{1/2}(\varphi_{-d/c})\;\frac{du}{\sqrt{\pi u}} = \left(\int_0^{e^{i(\frac{\pi}{2}-\varepsilon)}\infty}+ \int_0^{e^{i(\frac{\pi}{2}+\varepsilon)}\infty}\right) e^{-u/(z+d/c)}\mathcal{B}_{1/2}(\varphi_{-d/c})\;\frac{du}{\sqrt{\pi u}},
\]	
 noting that the integrand changes sign when $u$ circles the origin. This shows that this integral is the median resummation with $\varphi_{-d/c}$ of angle $\pi/2$. 

To prove Proposition \ref{prop: med resum}, it then suffices to show that in the given cases,
\begin{equation}\label{eqn: 0 integral}
\frac{1}{2\sqrt{z+d/c}}\int_{\Gamma}e^{-u/(z+d/c)}\widehat{g}_{-d/c}^{odd}(u)\frac{du}{\sqrt{u}} = 0.
\end{equation}
We will treat both cases separately. 

First, suppose that $g$ is the Eichler integral of a cusp form $f$ of weight $k \in \frac{1}{2}+\mathbb{Z}$ on $\Gamma_0(N)$ where $4 \mid N$. Then $b_n = a_n n^{1-k}$ for $f(z) = \sum_{n=1}^\infty a_nq^n$. 
Considering equation (\ref{egn: g hat asymptotic}), note that for $r/2 \geq k-1$ and $r$ odd, $L(f, \zeta_c^{-d};k-1-r/2)$ vanishes due to the poles of $\Gamma(\rho)$ and the fact that $\Lambda(f, \zeta_c^{-d};\rho)$ is entire. This means that 
\[
\frac{1}{\sqrt{u}}\widehat{g}_{-d/c}^{odd}(u) = -\frac{1}{\sqrt{\pi }} \sum_{n=0}^{k-5/2} \frac{ L(g, \zeta_c^{-d};-n-1/2)}{\Gamma(2n+2)}(8\pi i)^{n+1/2}u^{n},
\]
which is just a polynomial in $u$. Then, we claim that \[\int_{\Gamma} e^{-u/(z+d/c)}\left( \sum_{n=0}^{k-5/2} \frac{L(g, \zeta_c^{-d};-n-1/2)}{\Gamma(2n+2)}(8\pi i)^{n+1/2}u^{n}\right)\;du = 0.\]
To see this, one considers an arc of radius $R>0$ which connects two points of $\Gamma$ to form an oriented closed curve. The integral along this entire curve must be 0 by the residue theorem, and as $R$ tends to infinity, the integral along the arc approach zero since the expression in brackets is a polynomial and $e^{-u/(z+d/c)}$ decays exponentially as $\Im(u) \to \infty$ (recalling that $z+d/c \in i\mathbb{R}_+)$.

Now suppose that $g$ is of \emph{Type 5}. We consider only $z \to 0$. Showing Equation \ref{eqn: 0 integral} will require more work in this case. In order to show that the integral vanishes, we realize $\frac{1}{2\sqrt{u}}\widehat{g}^{odd}_{-d/c}$ as a Borel transform of another asymptotic series. Taking the Borel transform with a different shift results in a much simpler function which is easily shown to be holomorphic on $\mathbb{H}$. This, coupled with Lemma \ref{lem: borel laplace shifts}, is enough to prove our claim. 

Recalling the definitions of $L_0$ and $L_1$ from Section \ref{subsec: maass}, we have 
\[L_0(-n-1/2) =\frac{1}{\pi^{2n+3} \cdot 2^{2n+2}} \Gamma(n+3/2)^2L_0(n+3/2);
\]
\[L_1 (-n-1/2) = - \frac{1}{\pi^{2n+3} \cdot 2^{2n+2}} \Gamma(n+3/2)^2L_1(n+3/2).
\]

Returning to $\widehat{g}_0^{odd}$ in this new case, we have
\[\frac{1}{2\sqrt{u}}\widehat{g}_0^{odd}(u) = -\frac{1}{2\sqrt{\pi}}\sum_{n=0}^\infty \frac{(8\pi i)^{n+1/2}}{\pi^{2n+3} \cdot 2^{2n+2}} \frac{\Gamma(n+3/2)^2}{\Gamma(2n+2)} \left(L_0(n+3/2)-L_1(n+3/2)\right) u^n. 
\]
Noting that $\Gamma(2n+2) = \frac{2^{2n+1}}{\sqrt{\pi}} \Gamma(n+1)\Gamma(n+3/2)$, we can write this as 
\[ -\sum_{n=0}^\infty \frac{1}{\pi^{2n+3} \cdot 2^{2n+3}} \frac{\Gamma(n+3/2)}{\Gamma(n+1)} \left(L_0(n+3/2)-L_1(n+3/2)\right)(2\pi i)^{n+1/2} u^n .
\]

Now we define the function
\[\Phi(z) := -\sum_{n=0}^\infty \frac{1}{\pi^{2n+3} \cdot 2^{2n+3}} \Gamma(n+3/2) \left(L_0(n+3/2)-L_1(n+3/2)\right)(2\pi i)^{n+1/2} z^n.\]
Then $\mathcal{B}_1(\Phi(z)) = \frac{1}{2\sqrt{u}}\widehat{g}_0^{odd}(u)$. 
Let us instead consider $\mathcal{B}_{3/2}(\Phi(z))$. This is equal to 
\[-\sum_{n=0}^\infty \frac{1}{\pi^{2n+3} \cdot 2^{2n+3}} \left(L_0(n+3/2)-L_1(n+3/2)\right)(2\pi i)^{n+1/2} u^n.\]

From here, we compute
\begin{align*}
		\mathcal{B}_{3/2}(\Phi(z)) = &-\sum_{n=0}^\infty \frac{(2\pi i)^{n+1/2}}{\pi^{2n+3}2^{2n+3}} \left[ \sum_{m=1}^\infty \frac{A_{-m}}{m^{n+3/2}}\right]u^{n}\\
	= &-\frac{1}{2\pi i}\sum_{n=0}^\infty \left[ \sum_{m=1}^\infty \frac{A_{-m}}{(-2\pi im)^{n+3/2}}\right]u^{n}\\
	= &-\frac{1}{2\pi i}\sum_{m=1}^\infty \frac{A_{-m}}{(-2\pi i m)^{3/2}} \left[ \sum_{n=0}^\infty \frac{u^n}{(-2\pi im)^{n}}\right]\\
	= &\frac{1}{2\pi i}\sum_{m=1}^\infty \frac{A_{-m}}{(-2\pi i m)^{1/2}} \left( \frac{1}{u-(-2\pi i m)}\right)\\
\end{align*}
This function has only simple poles at points $-2\pi i m$ for $m \in \mathbb{N}$. Importantly, the poles are \textit{not} in the upper half-plane. 
In particular, we have that 
\[\mathcal{L}_{3/2}^{\pi/2-\varepsilon}(\mathcal{B}_{3/2}(\Phi(z)) = \mathcal{L}_{3/2}^{\pi/2+\varepsilon}(\mathcal{B}_{3/2}(\Phi(z))
\]
for $ 0< \varepsilon < \pi/2$. 
Then by Lemma \ref{lem: borel laplace shifts}, we have that 
\[\mathcal{L}_{1}^{\pi/2-\varepsilon}(\mathcal{B}_{1}(\Phi(z)) = \mathcal{L}_{3/2}^{\pi/2-\varepsilon}(\mathcal{B}_{3/2}(\Phi(z)) = \mathcal{L}_{3/2}^{\pi/2+\varepsilon}(\mathcal{B}_{3/2}(\Phi(z)) = \mathcal{L}_{1}^{\pi/2+\varepsilon}(\mathcal{B}_{1}(\Phi(z)).
\]

This is exactly what we want. To see this, consider again
\[\frac{1}{2}\int_{\Gamma}e^{-u/z}\widehat{g}_{0}^{odd}(u)\frac{du}{\sqrt{u}}.\]
For $\varepsilon$ sufficiently small, the contour $\Gamma$ in Figure \ref{fig: contours} is equivalent to a median resummation contour for $\theta=\pi/2$ since our function is analytic in a disc around the origin.
From here, we have that 

\begin{align*}\frac{1}{2}\int_\Gamma e^{-u/z} \widehat{g}_{0}^{odd}(u) \;\frac{du}{\sqrt{u}} &= \int_\Gamma e^{-u/z} \mathcal{B}_1(\Phi(z))\;du
\\
&= \left(\int_0^{e^{i(\frac{\pi}{2}-\varepsilon)}\infty}- \int_0^{e^{i(\frac{\pi}{2}+\varepsilon)}\infty}\right) e^{-u/z} \mathcal{B}_1(\Phi(z)) \;du \\
&=  z\left(\mathcal{L}_{3/2}^{\pi/2-\varepsilon}\mathcal{B}_{3/2}(\Phi(z)) - \mathcal{L}_{3/2}^{\pi/2+\varepsilon}\mathcal{B}_{3/2}(\Phi(z))\right) = 0. 
\end{align*}
Now we have shown that the second integral in Equation \ref{eqn: decomp integral} vanishes, proving that $g(z)
=  \mathcal{S}^{\pi/2}_{med}(\varphi_{-d/c})(z)$ as desired.
\qedhere

\end{proof}

In the case where $g$ is of \emph{Type 6}, we modify the argument slightly and end up with a similar result. 
\begin{prop} \label{prop: eisenstein med}
If $g$ is of Type 6, then \[g= \mathcal{S}_{med}^{\pi/2}(\varphi_{0}) +\frac{1}{2\sqrt{z}}\int_\Gamma e^{-u/z} \left( \frac{-4+2\log(8\pi i u)}{2\pi i u} \right)\frac{du}{\sqrt{\pi u}}.\]	
\end{prop}
\begin{proof}
We write 
\[1-4\sum_{n = 1}^\infty d(n) q^n =  \frac{1}{2\sqrt{z}}\int_{\Gamma}e^{-u/z} \left( 1-4\sum_{n = 1}^\infty d(n) e^{-\sqrt{8\pi i nu}}\right) \;\frac{du}{\sqrt{\pi u}},
\]
and using the Mellin transform again we rewrite the inner sum as  
\begin{equation*}  - \frac{4-2\log(8\pi i u)}{2\pi i u} - \frac{4}{\sqrt{\pi }} \sum_{r=1}^\infty \frac{(-1)^r \zeta^2(-r/2)}{\Gamma(r+1)}(8\pi iu)^{r/2},
\end{equation*}
Noting that the calculation for the pole at $s=2$ differs slightly.
Leaving the leading term, we again split the sum into the cases where $r$ is even or odd to get 
\[\widehat{g}_{0}^{even}(u) := \frac{-4}{\sqrt{\pi }} \sum_{n=0}^\infty \frac{ \zeta^2(-n)}{\Gamma(2n+1)}(8\pi iu)^{n};
\]
\[\frac{1}{\sqrt{u}}\widehat{g}_{0}^{odd}(u) := \frac{4}{\sqrt{\pi }} \sum_{n=0}^\infty \frac{ \zeta^2(-n-1/2)}{\Gamma(2n+2)}(8\pi i)^{n+1/2}u^{n}.
\]
As before, we have that $\widehat{g}_{0}^{even}(u) = \mathcal{B}_{1/2}(\varphi_0)$, where $\varphi_0$ was defined for $g$ in the proof of Proposition \ref{prop: eisenstein asymptotic}. The proof that the integral of $\widehat{g}_{0}^{odd}(u)$ vanishes is the same as in Proposition \ref{prop: med resum} when $g$ is of \emph{Type 5}. 
After appropriately rescaling to get from $\sum_{n=1}^\infty d(n)q^n$ to $g$, we get the desired result. 
\end{proof}

Note that in the cases where $g$ is of \emph{Type 1 , 2,} or \emph{3}, one can argue the same way as in Proposition \ref{prop: med resum} to get the same decomposition as in 
 Equation \ref{eqn: decomp integral}. However, in these cases, the second integral no longer vanishes. Moreover, we know that when $g$ is of \emph{Type 1} or \emph{2},  $\mathcal{B}_{1/2}(\varphi_{-d/c}(z))$ is zero, while for $g$ of \emph{Type 3}, it is a polynomial. So we have that the first integral satisfies
\begin{align*}&\frac{1}{2\sqrt{z+d/c}} \int_{\Gamma}e^{-u/(z+d/c)} \mathcal{B}_{1/2}(\varphi_{-d/c}(z)) \frac{du}{\sqrt{\pi u}}\\&= \frac{1}{\sqrt{z+d/c}} \int_{0}^{\infty}e^{-u/(z+d/c)} \mathcal{B}_{1/2}(\varphi_{-d/c}(z)) \frac{du}{\sqrt{\pi u}},\end{align*}
which up to prefactors is the Borel-Laplace sum of $\varphi_{-d/c}$ in direction 0. 

\subsection{Discontinuity Calculations} \label{subsec: discontinuity} Here we compute the discontinuities of the Borel-Laplace sums of the functions of \emph{Types 4} and \emph{5} along the ray of angle $\pi/2$. 

\begin{prop}\label{prop: eichler discont}
	If $g$ is of Type 4 with $2-k= \kappa$ or of Type 5 with $s = \kappa/2$, then  for $\gamma = \begin{psmallmatrix} a & b \\ c & d \end{psmallmatrix} \in \Gamma_0(N)$, we have that 
 $2g|_{\kappa}\gamma(z) = \emph{disc}_\kappa^{\pi/2}(\mathcal{B}_\kappa(\varphi_{-d/c}(z))$.
\end{prop}

\begin{proof}
We begin with $g$ of \emph{Type 4} and $\widehat{\varphi}_{-d/c}^{\ell}(z)$ as defined in 
Section \ref{subsec: asymptotics}. Note that $\mathcal{S}_1^{\pi/2 - \varepsilon}(\widehat{\varphi}_{-d/c}^{\ell})=\mathcal{S}_1^{\pi/2 + \varepsilon}(\widehat{\varphi}_{-d/c}^{\ell})$ for $0< \varepsilon< \pi/2$ since $\widehat{\varphi}_{-d/c}^{\ell}$ is a polynomial. Therefore, these terms cannot contribute to the discontinuity. 

By the computations in Section \ref{subsec: asymptotics}, we also have that $\mathcal{B}_{2-k}(\varphi^\ell_{-d/c}(z))$ has simple poles at $\omega_m = \frac{2\pi i m}{c^2}$ with residues $-2 \frac{\varepsilon_d^{2k} \left( \frac{c}{d}\right)}{c^{2-k}}  a_m m^{1-k}\zeta_c^{am}\omega_m^{k-1}$.
Then
\begin{align*}\text{disc}_{2-k}^{\pi/2}(\mathcal{B}_{2-k}(\varphi_{-d/c}(z)) &= (z+d/c)^{k-2}\sum_{\omega \in \Omega_\theta} S_\omega \omega^{1-k} e^{-\omega/(z+d/c)}\\
	&= 2 \frac{\varepsilon_d^{2k} \left( \frac{c}{d}\right)}{c^{2-k}} (z+d/c)^{k-2}\sum_{n =1}^\infty    a_m m^{1-k}\zeta_c^{am} e^{-2\pi i m/(c^2(z+d/c))}\\
	&= 2(cz+d)^{k-2}\varepsilon_d^{2k} \left( \frac{c}{d}\right)\sum_{n =1}^\infty    a_m m^{1-k} e^{2\pi i m\left(\frac{az+b}{cz+d}\right)} \\
	&= 2g|_{2-k}\gamma(z),
\end{align*}
where the second to last equality follows from the identity  $\frac{1}{c^2(\tau+d/c)} = \frac{a}{c} - \frac{a\tau+ b}{c\tau +d}$.
Turning to \emph{Type 5}, we have in this case that
\[\text{disc}_\kappa^{\pi/2}(\mathcal{B}_\kappa(\varphi_{0}(z)) = z^{-1}\sum_{\omega \in \Omega_\theta} S_\omega e^{-\omega/z} =2 z^{-1} \sum_{n=1}^\infty 	A_n e^{-2\pi i n/z} = 2g|_{1}(S)(z). \qedhere
\] 
\end{proof}

Using the analogous result for the Eisenstein case given in Proposition \ref{prop: eisenstein asymptotic}, we can also prove the following. 
\begin{prop}
	If $g$ is of Type 6, then  for $\gamma = \begin{psmallmatrix} a & b \\ c & d \end{psmallmatrix} \in \Gamma_0(N)$, we have that 
 $2g|_{1}\gamma(z) = \emph{disc}_1^{\pi/2}(\mathcal{B}_1(\varphi_{0}(z))$.
\end{prop}

Now we return to the cases that require a different approach. 
\begin{prop} \label{prop: modular slash}
If $g$ is of Type 1 or 2 with $\kappa = k$ or Type 3 with $\kappa = 2-k$, then for $\gamma = \begin{psmallmatrix} a & b \\ c & d \end{psmallmatrix} \in \Gamma_0(N)$, we have that 
\[g|_{\kappa}\gamma(z) = \frac{1}{2\sqrt{z+d/c}} \int_{\Gamma}e^{-u/(z+d/c)}\hat{g}_{-d/c}^{odd}(u)\frac{du}{\sqrt{u}}.\]
\end{prop}

\begin{proof}	
In the cases where $g$ is of \emph{Type 1} or \emph{2}, since the asymptotic series $\varphi_{-d/c}(z)$ is zero, we have that $\mathcal{B}_{1/2}(\varphi_{-d/c}(z))=0$. So, it suffices to analyze 
 \[\frac{1}{\sqrt{u}}\widehat{g}_{-d/c}^{odd}(u) = -\frac{1}{\sqrt{\pi }} \sum_{n=0}^\infty \frac{ L(g, \zeta_c^{-d};-n-1/2)}{\Gamma(2n+2)}(8\pi i)^{n+1/2}u^{n}.
\]
 
We first consider $g$ of \emph{Type 1} or \emph{2}. We assume the notation and some of the techniques of Proposition \ref{prop: med resum}. 
We can write the above as 

 \begin{align*}\frac{1}{\sqrt{u}}\widehat{g}_{-d/c}^{odd}(u) &= - \sum_{n=0}^\infty \frac{ L(g, \zeta_c^{-d};-n-1/2)}{\Gamma(n+1)\Gamma(n+3/2)}(2\pi i)^{n+1/2}u^{n}\\
 &=C_{k, -d/c} \frac{i^{-k}}{\pi} \sum_{n=0}^\infty \Big\lbrack \left(\frac{c}{2\pi}\right)^{k+2n+1}\frac{\Gamma(n+k+1/2)}{\Gamma(n+1)}\\ & \quad \quad \quad \quad\quad \quad\quad \quad\cdot L(g, \zeta_c^{a}; n+k+1/2) (-1)^n(2\pi i)^{n+1/2}u^n\Big\rbrack,
\end{align*}
where $C_{k, -d/c}=1$ when $k$ is an integer and equals $\varepsilon_d^{2k}\left(\frac{c}{d}\right)$ when $k$ is a half-integer.

If we define
\[\Phi(z) =C_{k, -d/c}\frac{i^{-k}}{\pi} \sum_{n=0}^\infty \left(\frac{c}{2\pi}\right)^{k+2n+1} (-1)^n\Gamma(n+k+1/2)L(g, \zeta_c^{a}; n+k+1/2)(2\pi i)^{n+1/2}z^n,
\]
then we see that $\mathcal{B}_1(\Phi(z)) = \frac{1}{\sqrt{u}}\widehat{g}_{-d/c}^{odd}(u)$. Instead consider $\mathcal{B}_{k+1/2}(\Phi(z))$, which equals 
\[i^{-k} C_{k, -d/c}\sum_{n=0}^\infty \left(\frac{c}{2\pi}\right)^{k+2n+1} (-1)^nL(g, \zeta_c^{a}; n+k+1/2)(2\pi i)^{n+1/2}u^n.
\]
Writing $L(g, \zeta_c^{a}; n+k+1/2)$ as its Dirichlet series and using absolute convergence allows us to write this as 
\begin{align*}&\frac{-C_{k, -d/c}}{\pi i c^k} \sum_{m=1}^\infty \sum_{n=0}^\infty \frac{a(m)\zeta_c^{am}}{(2\pi i m/c^2)^{n+k+1/2}}u^n = \frac{C_{k, -d/c}}{\pi i c^k}\sum_{m=1}^\infty  \frac{a(m)\zeta_c^{am}}{u - 2\pi i m/c^2}\frac{1}{(2\pi i m/c^2)^{k-1/2}}.
\end{align*}
We can argue along the same lines as in Proposition \ref{prop: med resum} in order to move between the Borel-Laplace sum of $\Phi$ with shift 1 and with shift $k+1/2$. We have that
\begin{align*}
&\frac{1}{2\sqrt{z+d/c}} \int_{\Gamma}e^{-u/(z+d/c)}\widehat{g}_{-d/c}^{odd}(u)\frac{du}{\sqrt{u}}\\
&= \frac{1}{2\sqrt{z+d/c}} \int_{\Gamma}e^{-u/(z+d/c)}\mathcal{B}_1(\Phi)\;du\\
&= \frac{1}{2}\sqrt{z+d/c}\left( \mathcal{L}_{1}^{\pi/2-\varepsilon}(\mathcal{B}_1(\Phi)) - \mathcal{L}_{1}^{\pi/2+\varepsilon}(\mathcal{B}_1(\Phi))\right) \\
&= \frac{1}{2}\sqrt{z+d/c}\left( \mathcal{L}_{k+1/2}^{\pi/2-\varepsilon}(\mathcal{B}_{k+1/2}(\Phi)) - \mathcal{L}_{k+1/2}^{\pi/2+\varepsilon}(\mathcal{B}_{k+1/2}(\Phi))\right)\\
&= \sqrt{z+d/c}\cdot c^{-k} (z+d/c)^{-k-1/2} \sum_{m=1}^\infty a(m)\zeta_c^{am}\frac{(2\pi i m/c^2)^{k-1/2}}{(2\pi i m/c^2)^{k-1/2}} e^{2\pi i m/c^2(z+d/c)}\\
&= (cz+d)^{-k}\sum_{m=1}^\infty a(m) e^{2\pi i m\left(\frac{az+b}{cz+d}\right)}.
\end{align*}

In the case where $g$ is of \emph{Type 3}, the argument is largely the same. We assume that $g$ is the Eichler integral of $f$. Turning our attention back to 
\[\frac{1}{\sqrt{u}}\widehat{g}_{-d/c}^{odd}(u) = -\frac{1}{\sqrt{\pi }} \sum_{n=0}^\infty \frac{ L(g, \zeta_c^{-d};-n-1/2)}{\Gamma(2n+2)}(8\pi i)^{n+1/2}u^{n},
\]
we see that this equals 
\begin{align*}
	&-\frac{1}{\sqrt{\pi}}\sum_{n=0}^\infty \frac{L(f, \zeta_c^{-d};-n+k-3/2)}{\Gamma(2n+2)}(8\pi i)^{n+1/2} u^{n} \\
	&= \frac{- c^{k-2} \varepsilon_d^{2k} \left( \frac{c}{d}\right)}{\pi i}\sum_{n=0}^\infty  \frac{\Gamma(n-k+5/2)}{\Gamma(n+1)} \frac{L(f, \zeta_c^{a}; n+3/2)}{(2\pi i/c^2)^{n+5/2-k}} u^n.\\
\end{align*}
Again we define 
\[\Phi(z) := \frac{- c^{k-2} \varepsilon_d^{2k} \left( \frac{c}{d}\right)}{\pi i}\sum_{n=0}^\infty \Gamma(n-k+5/2) \frac{L(f, \zeta_c^{a}; n+3/2)}{(2\pi i/c^2)^{n+5/2-k}} z^n,
\]
so that $\mathcal{B}_1(\Phi(z)) =\frac{1}{\sqrt{u}}\widehat{g}_{-d/c}^{odd}(u)$. After splitting off finitely many terms to ensure convergence of the Dirichlet series representation of $L(f, \zeta_c^{a}; n+3/2)$, we take the Borel transform of the remaining terms of $\Phi$ with shift $5/2-k$. 

We proceed as before to realize the discontinuity after shifting as in Proposition \ref{prop: eichler discont}, deforming contours and relating Borel-Laplace sums as for \emph{Types 1} and \emph{2}. \end{proof}

\section{Proof of Theorem \ref{thm: Main}}\label{sec: main theorem}
We now prove our main theorem. 
We proceed by analyzing each case in order. For $g$ of \emph{Type 1} or \emph{2}, by Section \ref{subsec: asymptotics}, we know that the Borel--Laplace sum 
\[\frac{1}{\sqrt{z+d/c}} \int_{0}^{\infty}e^{-u/(z+d/c)} \mathcal{B}_{1/2}(\varphi_{-d/c}(z)) \frac{du}{\sqrt{\pi u}}\]
equals zero. Since in these cases $g$ is modular, this equals their obstruction to modularity. Moreover, we have that
\[g(z) =  \frac{1}{2\sqrt{z+d/c}}\int_{\Gamma}e^{-u/(z+d/c)} \hat{g}_{-d/c}(u) \;\frac{du}{\sqrt{ u}} = \frac{1}{2\sqrt{z+d/c}} \int_{\Gamma}e^{-u/(z+d/c)}\hat{g}_{-d/c}^{odd}(u)\frac{du}{\sqrt{u}}.
\]
Lastly, by Proposition \ref{prop: modular slash}, we have that
\[g|_{\kappa}\gamma(z) = \frac{1}{2\sqrt{z+d/c}} \int_{\Gamma}e^{-u/(z+d/c)}\hat{g}_{-d/c}^{odd}(u)\frac{du}{\sqrt{u}}.\]
Thus we have that $g(z) = g|_{\kappa}\gamma(z)$. 

When $g$ is of \emph{Type 3}, the argument is essentially the same, except the integral
\[\frac{1}{\sqrt{z+d/c}} \int_{0}^{\infty}e^{-u/(z+d/c)} \mathcal{B}_{1/2}(\varphi_{-d/c}(z)) \frac{du}{\sqrt{\pi u}}\] is not zero. We note that this is precisely the Borel resummation of $\varphi_{d/c}$ with shift $1/2$, and so it is equal to the Borel resummation of $\varphi_{d/c}$ with shift $1$. We will show explicitly that this equals the expected period polynomial in Proposition {prop: polynomial}. 

For $g$ of \emph{Type 4} or \emph{5}, we have already shown that $g(z)=\mathcal{S}_{med}^{\pi/2}(\varphi_{-d/c})(z)$ in Proposition \ref{prop: med resum}. From Section \ref{subsec: resurgence}, we have that
\[\mathcal{S}_{med}^{\pi/2}(\varphi_{-d/c})=\mathcal{S}^{0}(\varphi_{-d/c}) + \frac{1}{2}\text{disc}_{\kappa}^{\pi/2}(\mathcal{B}_{\kappa}(\varphi_{-d/c}))\]
for $z+d/c$ in the first quadrant. By Proposition \ref{prop: eichler discont}, we have that 
\[\frac{1}{2}\text{disc}_{\kappa}^{\pi/2}(\mathcal{B}_{\kappa}(\varphi_{-d/c})) = g|_{\kappa}\gamma(z)
\]
where $\gamma = \begin{psmallmatrix} a & b \\ c & d
 \end{psmallmatrix}$. This tells us that $g(z) - g|_\kappa\gamma(z) = \mathcal{S}^0(\varphi_{-d/c})$ in this region. By analytic continuation, we get the desired equality on $\mathbb{H}$. The analytic continuation of $\mathcal{S}^0(\varphi_{-d/c})$ to $\mathbb{C}_\gamma$ is guaranteed by the location of the Stokes rays for $\mathcal{B}_\kappa (\varphi_{-d/c})$.
The case where $g$ is of \emph{Type 6} is identical except for the addition of the integral
\[\frac{1}{2\sqrt{z}}\int_\Gamma e^{-u/z} \left( \frac{-4+2\log(8\pi i u)}{2\pi i u} \right)\frac{du}{\sqrt{\pi u}},\]
which contributes to the obstruction to modularity.\hfill \qedsymbol

\section{Proofs of Corollaries}\label{sec: corollaries}

Here we show that the presentations that are more familiar for these modular obstructions are readily obtained from the Borel-Laplace sum representation. 
We first show that the modular obstruction in Corollary \ref{cor: eichler} has the stated form. This is all that is left to prove after proving Theorem \ref{thm: Main}. 

\begin{prop}\label{prop: obstruction}
Suppose $c>0$ and $\tilde{f}$ is of Type 4. Then we have that
\begin{align*}\mathcal{L}^0_1(\mathcal{B}_1(\varphi_{-d/c})) &=
  \frac{(-2\pi i)^{k-1}}{\Gamma(k-1)} \int_{-d/c}^{i\infty} f(\tau)\left(\tau- z\right)^{k-2}\;d\tau, \quad \quad (z > -d/c).
\end{align*} 

\end{prop}
\begin{proof}
	We begin with the asymptotic expansion 
\begin{align*}
\tilde{f}(z) \sim \varphi_{-d/c}(z) = \varepsilon_d^{2k} \left( \frac{c}{d}\right)   \sum_{n=0}^\infty \frac{(-1)^{n+1-k}}{\Gamma(k-1-n) c^{2-k}} \sum_{m=1}^\infty \frac{a_f(m)\zeta_c^{am}m^{1-k}}{\left( \frac{2\pi i m}{c^2}\right)^{n+2-k}}( z+d/c)^{n}.
\end{align*}
Computing the Borel transform with shift $1$ with respect to $z+d/c$ gives
\begin{align*}\mathcal{B}_1(\varphi_{-d/c})(u) 
&= \frac{\varepsilon_d^{2k} \left( \frac{c}{d}\right) }{\Gamma(k-1)} \sum_{m=1}^\infty \frac{a_f(m)\zeta_c^{am}m^{1-k}(-1)^{1-k}}{c^{2-k}\left( \frac{2\pi i m}{c^2}\right)^{2-k}} \sum_{n=0}^\infty \frac{\Gamma(k-1)}{\Gamma(k-1-n) \Gamma(n+1)}\frac{(-1)^nu^{n}}{\left( \frac{2\pi i m}{c^2}\right)^{n}}
\end{align*}
The second sum is equal to $(1- \frac{u}{2\pi im/c^2})^{k-2}$ for $|u| < 2\pi m/c^2$, and \[(2\pi i m/c^2)^{k-2} \left(1- \frac{u}{2\pi im/c^2}\right)^{k-2} = (2\pi i m/c^2 - u)^{k-2}\] when $-\pi/2 < \arg(u) < \pi$. Then the above equals
\begin{align*}
\frac{\varepsilon_d^{2k} \left( \frac{c}{d}\right) }{\Gamma(k-1)}  \sum_{m=1}^\infty \frac{a_f(m)\zeta_c^{am}m^{1-k}(-1)^{1-k}}{c^{2-k}\left( \frac{2\pi i m}{c^2}- u \right)^{2-k}} .\end{align*}	
At this point, we compute $\mathcal{L}^0_1(u)(t)$. After a change of variables, this is given by an integral expression of the form
\begin{align*}
	\int_0^\infty e^{-u} \mathcal{B}_1(\varphi_{-d/c})(ut)\;du, 
\end{align*}
where $t>0$. We compute
\begin{align*}
	\int_0^\infty e^{-u} \mathcal{B}_1(\varphi_{-d/c})(ut)\;du &= \int_0^\infty e^{-u} \frac{\varepsilon_d^{2k} \left( \frac{c}{d}\right)  }{\Gamma(k-1)}  \sum_{m=1}^\infty \frac{a_f(m)\zeta_c^{am}m^{1-k}(-1)^{1-k}}{c^{2-k}\left( \frac{2\pi i m}{c^2}- u t\right)^{2-k}}\;du\\
	&=\frac{\varepsilon_d^{2k} \left( \frac{c}{d}\right)  (-1)^{1-k}}{c^{2-k}\Gamma(k-1)}\sum_{m=1}^\infty a_f(m) \zeta_c^{am}m^{1-k}\int_{0}^\infty \frac{e^{-u}}{\left( \frac{2\pi i m}{c^2}- u t\right)^{2-k} }\;du.
\end{align*}
We now switch the order of integration and summation, and after making the change of variables $u = -2\pi i m (\tau-a/c)$ in each integral, we obtain
\begin{align*}
&\frac{\varepsilon_d^{2k} \left( \frac{c}{d}\right)  (-1)^{1-k}}{c^{2-k}\Gamma(k-1)}\sum_{m=1}^\infty a_f(m) \zeta_c^{am}m^{1-k}(-2\pi i m) \int_{0}^{i\infty} \frac{e^{2\pi i m (\tau-a/c)}}{\left( \frac{2\pi i m}{c^2}+ 2\pi i m(\tau-a/c) t\right)^{2-k} }\;d\tau	\\
&= \frac{\varepsilon_d^{2k} \left( \frac{c}{d}\right)  (-1)^{1-k}}{c^{2-k}\Gamma(k-1)}\sum_{m=1}^\infty a_f(m) \frac{-2\pi i m^{2-k}}{(2\pi i m)^{2-k}} \int_{a/c}^{i\infty} \frac{e^{2\pi i m \tau}}{(\tau - a/c)^{2-k}\left(\frac{ 1}{c^2(\tau-a/c)}+ t\right)^{2-k} }\;d\tau.	
\end{align*}
Using that $\frac{1}{c^2(\tau- a/c)} = \frac{-d}{c} - \frac{d\tau -b}{-c\tau +a}$ and changing variables $t=z+d/c$, we obtain \begin{align*}
 \frac{\varepsilon_d^{2k} \left( \frac{c}{d}\right)}{\Gamma(k-1)}\sum_{m=1}^\infty \frac{ a_f(m)}{(2\pi i)^{1-k}} \int_{a/c}^{i\infty} \frac{(-c\tau + a)^{k-2} e^{2\pi i m \tau}}{\left(-\frac{d\tau -b}{-c\tau+a}+ z\right)^{2-k} }\;d\tau,	
\end{align*}
noting that collecting $-1/c$ with $\tau+a/c$ under the square root introduced a sign. Interchanging the order of integration and summation once again, the above becomes
\begin{align*}
\frac{(2\pi i)^{k-1}}{\Gamma(k-1)}\varepsilon_d^{2k} \left( \frac{c}{d}\right)  \int_{a/c}^{i\infty} \frac{(-c\tau + a)^{k-2} \sum_{m=1}^\infty a_f(m) e^{2\pi i m \tau}}{\left(-(\frac{d\tau -b}{-c\tau+a}- z\right))^{2-k} }\;d\tau,
\end{align*}
which under the change of variables $\tau \mapsto \gamma(\tau)$ gives
\begin{align*}
\frac{(2\pi i)^{k-1}}{\Gamma(k-1)} \int_{i\infty}^{\gamma^{-1}(i\infty)} \frac{\varepsilon_d^{2k} \left( \frac{c}{d}\right) (c\tau +d)^{2-k} \sum_{m=1}^\infty a_f(m) e^{2\pi i m \gamma(\tau)}}{\left(-(\tau - z\right))^{2-k} }\;\frac{d\tau}{(c\tau +d)^2}.
\end{align*}
Using the fact that $f$ is a modular form of weight $k$, as well as introducing an extra sign from pulling the $-1$ from under the square root, this becomes
\[
 \frac{(-2\pi i)^{k-1}}{\Gamma(k-1)} \int_{-d/c}^{i\infty} f(\tau)\left(\tau- z\right)^{k-2}\;d\tau. \qedhere
\]
\end{proof}

\

The same argument as above, taking into consideration differing signs at a few steps, allows us to also show the following.

\begin{prop}
Suppose $c<0$ and $\tilde{f}$ is of Type 4. Then we have that
\[\mathcal{L}^\pi_1(\mathcal{B}_1(\varphi_{-d/c})) =
 \frac{(-2\pi i)^{k-1}}{\Gamma(k-1)} \int_{-d/c}^{i\infty} f(\tau)\left(\tau- z\right)^{k-2}\;d\tau, \quad \quad (z > -d/c).
\] 
\end{prop}

Corollary \ref{cor: eichler} follows from the above since by Theorem \ref{thm: Main}, the obstruction to modularity equals the Borel-Laplace sum $\mathcal{L}^0_1(\mathcal{B}_1(\varphi_{-d/c}))$ or $\mathcal{L}^\pi_1(\mathcal{B}_1(\varphi_{-d/c}))$ depending on the sign of $c$. 

We remark that the above arguments also work in the case of integer weight Eichler integrals. However, the following shows how we can see the coefficients of the period polynomial explicitly from our resurgence methods. 

\begin{prop} \label{prop: polynomial}
	Suppose $g$ is of Type 3 for $\text{SL}_2(\mathbb{Z})$.
	Then if $g$ is the Eichler integral of $f$, we have that
\[\mathcal{L}^0_1(\mathcal{B}_1(\varphi_{0}))(z) = \frac{(-2\pi i)^{k-1}}{\Gamma(k-1)} \sum_{n=0}^{k-2}{k-2\choose n} i^{n-1}\Lambda(f; n+1)z^n. \]
\end{prop}

\begin{proof} Our work in Section \ref{subsec: asymptotics} computed $\varphi_0(z)$ for $g$ of \emph{Type 3} to be 
\[\sum_{n=0}^{k-2}  \frac{L(f; n+1)i^{n-k}}{(2\pi)^{n+2-k}{\Gamma(k-1-n)} } z^n.
\]
Since this is just a polynomial, it is equal to its own Borel-Laplace sum, and so we can manipulate it directly. Multiplying by
\[\frac{\Gamma(k-1)\Gamma(n+1)}{\Gamma(k-1)\Gamma(n+1)}
\]
and collecting terms gives the desired result.  
\end{proof}

\

\section{Generalizations and Open Questions}\label{sec: questions}
We conclude with a selection of follow-up questions of interest to the authors. 

\begin{enumerate}
\item In the cases where $g$ is of \emph{Type 5} or \emph{6}, we expect the natural extensions of Theorem \ref{thm: Main} to hold for general level and spectral parameter. However, the computations become more tedious when these extra considerations are introduced. 
\item We remark that our methods also extend to vector-valued settings. This allows for new proofs of related results, such as Theorem 1.5 in \cite{bringmann-nazaroglu}. 
\item From our results, it seems like $L$-functions of degree two are the natural objects for the proposed framework of Fantini and Rella \cite{fantini-rella}. The $L$-functions having degree two ensures that their $\Gamma$-factors give rise to Gevrey-1 asymptotic series. A natural question is whether one can use these same methods to study generating functions for higher degree $L$-functions by using Borel and Laplace transforms of higher index. 
\item It is rather interesting to the authors that resurgence on its face cannot fully encapsulate the simplest examples of quantum modular forms, e.g. honest modular forms and integer weight Eichler integrals. Here, we propose one method to deform these series which admit convergent asymptotic expansions in order to leverage techniques from resurgence analysis. In \cite{fantini-rella}, the authors construct ``modular resurgent completions" for cusp forms which break their modular invariance. Are there another methods which are better suited to unify all examples which use resurgence in some way?
\end{enumerate}

\bibliography{EM_LR_QMFs-and-Resurgence.bib}{}
\bibliographystyle{plain}

\end{document}